\documentclass[12pt]{amsart}

\usepackage[utf8]{inputenc}
\usepackage[english]{babel}
\usepackage{amsmath}
\usepackage{amsfonts}
\usepackage{amssymb}
\usepackage{graphicx}
\usepackage{mathtools}
\usepackage{amsthm}
\newtheorem*{Thm}{Theorem}

\newtheorem{thm}{Theorem}[section]
\newtheorem{prop}[thm]{Proposition}
\newtheorem{lem}[thm]{Lemma}

\newtheorem{rem}[thm]{Remark}

\theoremstyle{definition}
\newtheorem{defi}{Definition}[section]
\newcommand{\R}{\mathbb{R}}
\renewcommand{\P}{\mathbb{P}}

\newcommand{\Z}{\mathbb{Z}}
\newcommand{\C}{\mathbb{C}}
\newcommand{\Y}{\mathbb{Y}}

\newcommand{\E}{\mathbb{E}}
\newcommand{\Conf}{\mathrm{Conf}}
\newcommand{\dilog}{\text{dilog}}
\newcommand{\K}{\mathcal{K}}
\newcommand{\T}{\mathbb{T}}

\usepackage{fullpage}

\usepackage{hyperref}

\title{A law of large numbers for local patterns in Schur measures and a Schur process}
\author[P. Lazag]{Pierre Lazag}
\address{CNRS, UMR 7373, \'Ecole centrale de Marseille, Institut de Math\'ematiques de Marseille, Aix-Marseille Universit\'e, Marseille, France}
\address{SISSA, via Bonomea 265, 34136, Trieste, Italy}
\date{}
\begin{document}

\begin{abstract}
The aim of this note is to prove a law of large numbers for local patterns in discrete point processes. We investigate two different situations: a class of point processes on the one dimensional lattice including certain Schur measures, and a model of random plane partitions, introduced by Okounkov and Reshetikhin. The results state in both cases that the linear statistic of a function, weighted by the appearance of a fixed pattern in the random configuration and conveniently normalized, converges to the deterministic integral of that function weighted by the expectation with respect to the limit process of the appearance of the pattern.
\end{abstract}
\maketitle
\section{Introduction}
Our main results, Theorems \ref{thm:schur} and \ref{thm:ppp} below, are weak laws of large numbers for two different but related examples of discrete determinantal point processes: the Schur measures and a model of random plane partitions. The latter is an example of a Schur process, which are two-dimensional and dynamical generalizations of Schur measures. We start by describing a common and general framework encompassing both situations, before moving into details in Sections \ref{sec:introllnschur} and \ref{sec:intropp}. 
\subsection{A general framework}  \label{sec:intro}
\subsubsection{Description of the models}
We let $E$ be the state space. In our examples, we have either $E=\Z$ or $E=\Z \times \frac{1}{2}\Z$, and we simply assume here that $E$ is a discrete (i.e. without accumulation point) countable subset of $\R^d$ for a given integer $d\geq 1$. The probability space is the space of configurations $\Conf(E):= \{0,1\}^E$, equipped with the usual Borel sigma-algebra generated by the cylinders. We identify the set $\Conf(E)$ with the set of all subsets of $E$. A point process on $E$ is a probability measure on $\Conf(E)$. For a finite subset $m \subset E$, called a \emph{pattern}, we define the random variable
\begin{align*}
c_m : \Conf(E) &\rightarrow \R \\
X &\mapsto \left\{ \begin{matrix}
 1 & \text{if $m \subset X$,} \\
 0 & \text{else.}
\end{matrix} \right.
\end{align*}
Observe that by the inclusion/exclusion principle, the law of the random variables $c_m$, $ m \subset E$ is a pattern, characterizes a given point process. For $x \in \R^d$ and a pattern $m = \{m_1,\dots,m_l\}$, we write $m+ x := \{\lfloor m_1 + x \rfloor, \dots, \lfloor m_l + x \rfloor \}$, where $\lfloor y \rfloor$ is one of the closest point to $y \in \R^d$ in $E$. We consider a one parameter family of point processes $(\P_\alpha)_{\alpha >0}$ which admit a local limit behavior: there exists a set $A \subset \R^d$ and a family of point processes $(\P_{\mathcal{S}(u)})_{u \in A}$ such that the following holds
\begin{align} \label{lim:geneloc}
\lim_{\alpha \rightarrow + \infty} \E_{\alpha}  [ c_{ m + \alpha u} ] = \E_{\mathcal{S}(u)}  [ c_m ],
\end{align}
for all $u \in A$ and all finite $m\subset E$, where $\E_\alpha$ (resp. $\E_{\mathcal{S}(u)}$) denotes the expectation with respect to $\P_\alpha$ (resp. $\P_{\mathcal{S}(u)}$). The convergence (\ref{lim:geneloc}) can be interpreted as follows: if, for large $\alpha>0$, we zoom locally around the position $\alpha u$ in a configuration distributed according to $\P_\alpha$, the configuration we see behaves as it were distributed according to $\P_{\mathcal{S}(u)}$.
\subsubsection{The result}
For a compactly supported continuous function $f : \R^d \rightarrow \R$ and a finite subset $m \subset E$, we form the sum
\begin{align} \label{eq:geneempiraverage}
\Sigma ( f, m, \alpha) := \frac{1}{\alpha^d} \sum_{x \in \alpha A \cap E} f \left( \frac{1}{\alpha} x \right) c_{m + x}.
\end{align}
Consider the deterministic integral
\begin{align} \label{eq:geneint}
I(f,m) := \int_{A} f(u) \E_{\mathcal{S}(u)} [c_m] du.
\end{align}
A general Theorem in this context is the following.
\begin{Thm}Assume that for all compact $\mathcal{K} \subset A$, we have
\begin{align} \label{lim:genelocrest}
\alpha^d\sup_{u \in \mathcal{K}} \left| \E_\alpha [c_{m + \alpha u} ] - \E_{\mathcal{S}(u)} [c_m] \right| \rightarrow 0
\end{align}
as $\alpha \rightarrow + \infty$ and that
\begin{align} \label{lim:genevar}
\mathrm{var}_{\alpha} \left(\Sigma ( f, m, \alpha) \right) := \E_{\alpha} (\Sigma ( f, m, \alpha)- \E_{\alpha} \Sigma ( f, m, \alpha))^2  \rightarrow 0
\end{align}
as $\alpha \rightarrow + \infty$. Then for all $\epsilon >0$, we have
\begin{align} \label{thm:gene}
\lim_{ \alpha \rightarrow + \infty} \P_\alpha \left( \left| \Sigma(f,m,\alpha) - I(f,m) \right| > \epsilon \right) = 0.
\end{align}
\end{Thm}
\subsubsection{Scketch of the proof and some comments}
The proof is as follows. From (\ref{lim:geneloc}), we have
\begin{align*}
\E_\alpha \Sigma ( f, m, \alpha) = \frac{1}{\alpha^d} \sum_{x \in \alpha A \cap E } f \left( \frac{1}{\alpha} x \right)\E_{\mathcal{S}(x/\alpha)} [c_m] +o(1)
\end{align*}
where the $o(1)$ term arises from condition (\ref{lim:genelocrest}), since the function $f$ has compact support. One recognizes on the right hand side a Riemann sum for the integral $I(f,m)$. By condition (\ref{lim:genevar}), Chebyshev inequality implies the weak law of large numbers (\ref{thm:gene}).\\

We will prove in great details the same weak law of large numbers for a class Schur measures, Theorem \ref{thm:schur} below, and for a model of random plane partitions, Theorem \ref{thm:ppp} below. The main points will be to prove that conditions (\ref{lim:genelocrest}) and (\ref{lim:genevar}) are satisfied. Our proofs will lie on the fact that both point processes are determinantal point processes, with kernels described by double contour integrals that will reveal to be suitable for asymptotic analysis. We refer to Section \ref{sec:dpp} for the definition of determinantal point processes.\\

Both models we consider are models of random partitions. Other general laws of large numbers for discrete models related to random partitions have been established for example in \cite{bufetovgorin}, but in the case when the pattern $m$ is empty. The latter may thus be seen as global laws of large numbers, while Theorems \ref{thm:schur} and \ref{thm:ppp} below involve a mixture of global and local asymptotic  behaviors.

\subsection{The law of large numbers for Schur measures} \label{sec:introllnschur}
We start in Section \ref{sec:schurmeasuresorigin} by recalling the definition of the Schur measures on the set of partitions, following \cite{okounkovschurmeas}, see also \cite{borodingorin}, \cite{borodinrains} and \cite{johansson} for an other approach. We introduce a one parameter family of symmetric Schur measures and describe their local limit behavior in Section \ref{sec:schurmeaslim}. The law of large numbers is stated in Section \ref{sec:schurmeaslln}.
\subsubsection{The Schur measures on the set of partitions} \label{sec:schurmeasuresorigin}
Let $\Lambda$ be the algebra over $\C$ of symmetric functions, i.e. of symmetric polynomials with an infinite number of variables, see \cite{macdonald}. A partition $\lambda = (\lambda_1 \geq \lambda_2 \geq \dots )$ is an almost-zero sequence of non-negative integers, and is identified with a Young diagram. The set of all partitions is denoted by $\Y$. A distinguished basis in $\Lambda$ is formed by the Schur functions $s_\lambda$, indexed by partitions $\lambda \in \Lambda$. A specialization is an algebra morphism from $\Lambda$ to $\C$. A specialization $\rho$ is said to be \emph{Schur positive} if $\rho(s_\lambda) \geq 0$ for all $\lambda \in \Y$. The Schur measures are defined as follows.
\begin{defi}Let $\rho$, $\rho'$ be Schur positive specializations. The Schur measure $\mathbb{S}_{\rho,\rho'}$ is a probability measure on $\Y$ defined by 
\begin{align*}
\mathbb{S}_{\rho,\rho'}(\lambda)= C_{\rho,\rho'} \rho( s_\lambda) \rho' (s_\lambda), \quad \lambda \in \mathbb{Y}.
\end{align*}
\end{defi}
The normalizing constant $C_{\rho,\rho'}$ is given by the Cauchy formula (see \cite{macdonald})
\begin{multline*}
C_{\rho,\rho'}^{-1}= \sum_{ \lambda \in \Y} \rho (s_\lambda) \rho' (s_\lambda) = \rho \otimes \rho' \left( \prod_{i,j=1}^{+\infty} \frac{1}{1 - \mathrm{x}_i \mathrm{y}_j} \right)\\
= \rho \otimes \rho' \left( \exp \left( \sum_{k \geq 1} \frac{p_k(\mathrm{x}_1,\mathrm{x_2},\dots)p_k(\mathrm{y}_1,\mathrm{y_2},\dots)}{k}\right) \right),
\end{multline*}
where th $p_k$ are the Newton power sums (see \cite{macdonald}).

We assumed that the specializations $\rho$ and $\rho'$ are Schur positive in order to guarantee that the Schur measure $\mathbb{S}_{\rho,\rho'}$ is a positive measure. Observe now that this assumption might be avoided if one takes two complex conjugated specializations, which leads to the following definition.
\begin{defi}Let $\rho$ be a specialization. The \emph{symmetric Schur measure} $\mathbb{S}_\rho$ associated to $\rho$ is the probability measure on $\mathbb{Y}$ defined by
\begin{align*}
\mathbb{S}_{\rho} ( \lambda) = C_{\rho} |\rho ( s_\lambda) |^2, \quad \lambda \in \mathbb{Y}.
\end{align*}
\end{defi}
Again, the normalizing constant is given by the Cauchy formula
\begin{align*}
C_\rho^{-1} = \rho \otimes \overline{\rho}  \left( \prod_{i,j=1}^{+ \infty} \frac{1}{1-\mathrm{x}_i\mathrm{y}_j} \right).
\end{align*}
To a partition $\lambda =(\lambda_1 \geq \lambda_2 \geq \dots) \in \Y$, we associate a configuration $\mathfrak{S}(\lambda) \in \Conf(\Z)$ by
\begin{align*}
\mathfrak{S}(\lambda) = \{ \lambda_i - i, \hspace{0.1cm} i=1,2,\dots \}. 
\end{align*}
See Figure 1 below for a picture of the map $\mathfrak{S}$.
\begin{figure}
\begin{center}
\includegraphics[scale=0.75,  clip=true]{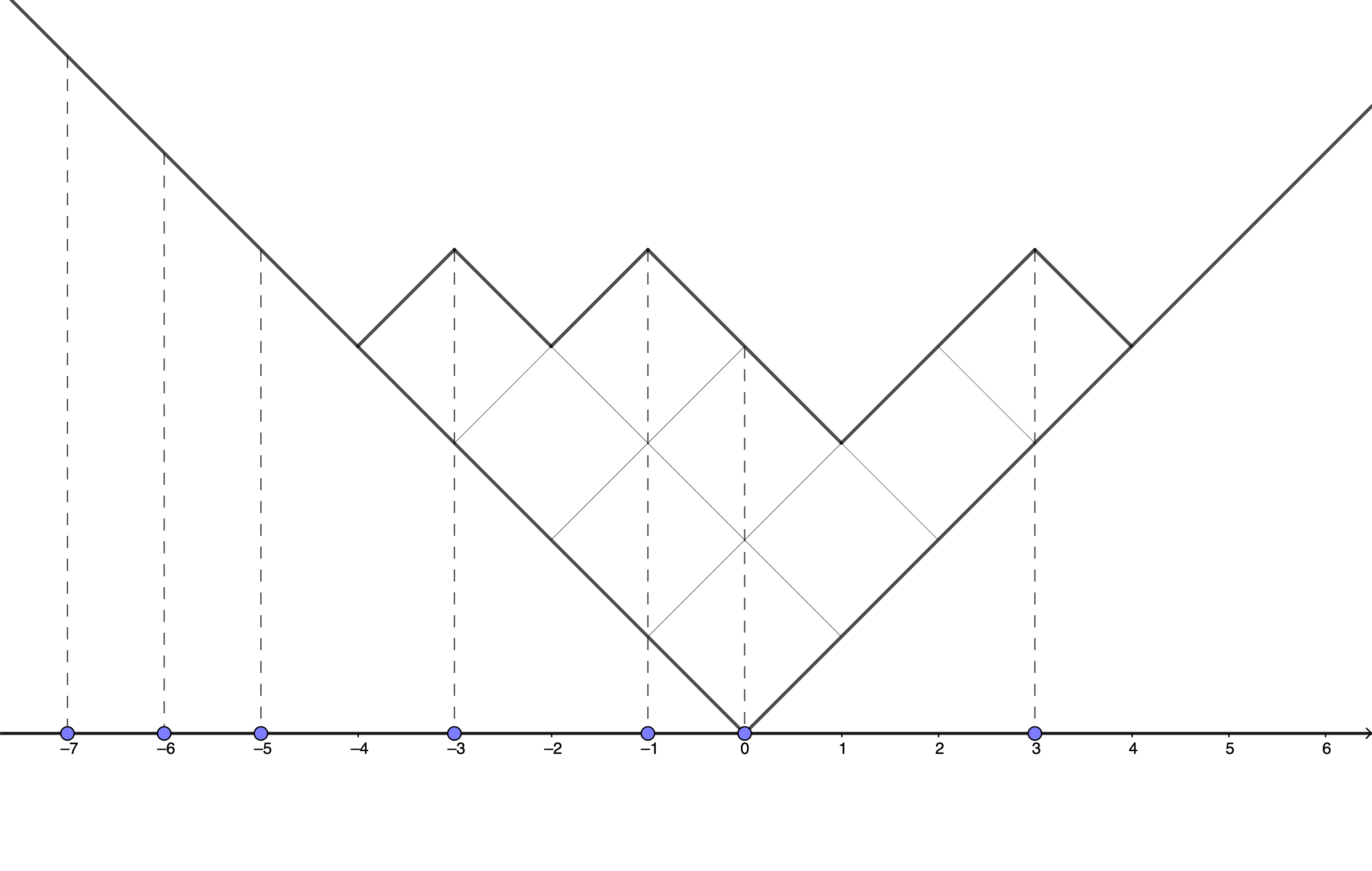}
\caption{The Young diagramm $(4,2,2,1)$
and its associated configuration.}
\end{center}
\end{figure}
Okounkov proved in \cite{okounkovschurmeas} that the image of the Schur measures by the map $\mathfrak{S}$ form determinantal point processes on $\Z$, see also \cite{borodingorin} and \cite{borodinrains} for other proofs. We recall this fact in the following Theorem, for the case of symmetric Schur measures.
\begin{thm}[Okounkov, \cite{okounkovschurmeas}] \label{thm:okounkovschurdet}The image of the symmetric Schur measure $\mathbb{S}_{\rho}$ by $\mathfrak{S}$, denoted by $\P_\rho$, is a determinantal point process on $\Z$. If the series
\begin{align*}
\sum_{k \geq 1} \frac{1}{k} \rho(p_k)z^k
\end{align*}
defines a holomorphic function in a neighborhood of the unit circle $\T := \{ z\in \C, \hspace{0.1cm} |z|=1\}$, then $\P_\rho$ admits the correlation kernel
\begin{multline} \label{eq:kernelschur}
K_\rho(x,y)=\\
= \frac{1}{(2i \pi)^2} \int_{|z|=1+ \varepsilon} \int_{|w|= 1 - \varepsilon} \exp \left(\sum_{k=1}^{+ \infty} \frac{\rho(p_k)}{k}(z^k- w^{k})+\frac{\overline{\rho(p_k)}}{k}(w^{-k}-z^{-k}) \right) \frac{1}{z-w}\frac{dzdw}{z^x w^{-y+1}}, \\
 x,y \in \Z,
\end{multline}
where $\varepsilon>0$ is sufficiently small.
\end{thm}

\subsubsection{The limit processes} \label{sec:schurmeaslim}
Our aim here is to state a local limit behavior for a certain class of Schur measures we now define. Let $\rho$ be a specialization such that the series
\begin{align*}
\sum_{k \geq 1} \frac{\rho(p_k)z^k}{k}
\end{align*}
defines a holomorphic function in a neighborhood of the unit circle $\T$. Let $G(z)$ be the function defined by
\begin{align} \label{eq:defg}
G(z)= \sum_{k \geq 1} \frac{\rho(p_k)z^k}{k} - \sum_{k \geq 1} \frac{\overline{\rho(p_k)}z^{-k}}{k}.
\end{align}
The function $G$ is holomorphic in a neighborhood of $\T$ and we have $G(z) \in i \R$ for all $z \in \T$. We let $\alpha>0$ be a positive parameter, and define the specialization $\rho_\alpha:= \alpha \rho$. We denote by $\mathbb{S}_\alpha$ the symmetric Schur measure with specialization $\rho_\alpha$ and let $\P_\alpha$ be its image on $\Conf(\Z)$ by the map $\mathfrak{S}$.

We can now state a local limit behavior for the point process $\P_\alpha$, as $\alpha \to + \infty$, Proposition \ref{prop:limschur} below, analogous to the convergence (\ref{lim:geneloc}) and (\ref{lim:genelocrest}).

Observe that $zG'(z) \in \R$ when $z \in \T$. Let  $u_{\min}$ (resp. $u_{\max}$) be the minimum (resp. maximum) of the function $zG'(z)$, $ z \in \T$. For $u \in [u_{\min},u_{\max}]$, the set
\begin{align*}
\{z \in \T, \hspace{0.1cm} zG'(z) \geq u \}
\end{align*}
consists of a finite number $L=L(u)$ of arcs $[e^{i\phi_k(u)},e^{i\psi_k(u)} ] \subset \T$, $k=1,\dots,L(u)$.  For $u \in [u_{\min},u_{\max}]$, we define the kernel
\begin{align} \label{eq:defsu}
\mathcal{S}(u)(x,y)&=\sum_{k=1}^{L(u)} \frac{e^{i\phi_k(u)(y-x)}-e^{i\psi_k(u)(y-x)}}{2i\pi(x-y)}, \quad x,y \in \Z, \hspace{0.1cm} x \neq y \\
\mathcal{S}(u)(x,x)&= \sum_{k=1}^{L(u)}\frac{\phi_k(u)-\psi_k(u)}{2\pi}, \quad x \in \Z
\end{align}
and let $\P_{\mathcal{S}(u)}$ be the determinantal point process with correlation kernel $\mathcal{S}(u)$. The local limit behavior (\ref{lim:geneloc}) together with the estimate (\ref{lim:genelocrest}) for the Schur measures $\P_\alpha$ are provided by the following Proposition. We denote by $\E_\alpha$ the expectation with respect to $\P_\alpha$, and $\E_{\mathcal{S}(u)}$ stands for the expectation with respect to $\P_{\mathcal{S}(u)}$. 
\begin{prop}\label{prop:limschur} For all $\delta \in (0,1)$ and all pattern $m\subset \Z$, there exists $C>0$ such that for all $u \in [u_{\min},u_{\max}]$,
\begin{align} \label{limprop:limschur}
\left| \E_\alpha [c_{m + \alpha u} ] - \E_{\mathcal{S}(u)} [c_m ] \right| \leq C \exp \left( - \alpha^\delta \right)
\end{align}
for all sufficiently large $\alpha>0$.
\end{prop}

Observe that if we have $L(u)=1$, the kernel $\mathcal{S}(u)$ defined in Equation (\ref{eq:defsu}) above reads
\begin{align*}
\mathcal{S}(u)(x,y) = \frac{e^{i\phi(u)(y-x)} - e^{i \psi(u) (y-x) }}{2i \pi(x-y)} = e^{i\frac{1}{2}(\phi(u) + \psi(u))(y-x)} \frac{\sin\left( \frac{\psi(u)-\phi(u)}{2}(x-y) \right)}{\pi(x-y)}.
\end{align*}
Since the factor  $e^{i\frac{1}{2}(\phi(u) + \psi(u))(y-x)}$ can be ignored (see Remark \ref{rem:gaugetransform}), the point process $\P_\alpha$ converges then to a version of the discrete sine process (see e.g. \cite{boo}, \cite{borodingorin}).\\

The class of Schur measures $\mathbb{S}_\alpha$ we considere here above are also discussed by Okounkov in the notes \cite{okounkovuses}. The limit Theorem given by Proposition \ref{prop:limschur} is also stated in the same paper \cite{okounkovuses} without the error term, and a proof is sketched. We give a complete detailed proof of Proposition \ref{prop:limschur} in Section \ref{sec:schurmeaslim} below.\\

Taking $G(z)=z - z^{-1}$ in equation (\ref{eq:defg}) above, one obtains the discrete Bessel point process, the corresponding Schur measure being the poissonized Plancherel measure (see \cite{boo}, \cite{borodinbessel}, \cite{johanssonplancherel} or \cite{borodingorin}). In particular, Proposition \ref{prop:limschur} is in that case the convergence of the discrete Bessel point process to the discrete sine process, first established in \cite{boo}.\\

In Section \ref{sec:rkshift}, we give an interpretation of the symmetric Schur measures and of the limit kernels $\mathcal{S}(u)$ in terms of shift invariant subspaces of $l^2(\Z)$. Namely, we prove that, up to minor transformations, the correlation kernels of symmetric Schur measures are projection kernels onto simply shift invariant subspaces, while the kernels $S(u)$ are projection kernels onto doubly shift invariant subspaces. Such an interpretation is also given in \cite{bufetovolshanski}. Proposition \ref{prop:limschur} provides a connection between both kinds of invariant subspaces.

\subsubsection{The law of large numbers} \label{sec:schurmeaslln}
As before, we set
\begin{align*}
\Sigma(f,m,\alpha)=\frac{1}{\alpha}\sum_{x \in [\alpha u_{\min},\alpha u_{\max}] \cap \Z} f \left( \frac{x}{\alpha} \right) c_{m+x}
\end{align*}
where $f :\R \rightarrow \R$ is a continuous function and $ m \subset \Z$ is a pattern. We also define the integral
\begin{align*}
I(f,m)= \int_{u_{\min}}^{u_{\max}} f(u) \E_{\mathcal{S}(u)} [c_m] du.
\end{align*}
The law of large numbers for Schur measures is the following Theorem.
\begin{thm}\label{thm:schur} For any continuous function $f : \R \rightarrow \R$ and any pattern $m \subset \Z$, we have for all $\epsilon >0$ that
\begin{align} \label{lim:llnschur}
\lim_{\alpha \rightarrow +\infty} \P_{\alpha} \left( \left| \Sigma(f,m,\alpha) - I(f,m) \right| > \epsilon \right) = 0.
\end{align}
\end{thm}

Theorem \ref{thm:schur} was first established for the discrete Bessel point process in \cite[Lem. 4.4]{bufetovgafa}, by means of simple estimations of the Bessel functions. The goal of the author in \cite{bufetovgafa} was to prove the Vershik-Kerov entropy conjecture for the Plancherel measure, and it would be intersting to see if our result could serve to define the entropy of more general Schur measures.

\subsection{The law of large numbers for plane partitions} \label{sec:intropp}

We now prepare the statement of the law of large numbers for plane partitions, Theorem \ref{thm:ppp} below. We present the model in Section \ref{sec:intromodelpp}, and give Okounkov-Reshtikhin determinantal formula in Section \ref{sec:ppdpp}. The limit processes are presented in Section \ref{sec:limpp}, and the law of large numbers is stated in Section \ref{sec:llnpp}.

\subsubsection{Introduction of the model} \label{sec:intromodelpp}
 A plane partition is a double non-increasing sequence of non-negative integers whith a finite number of non-zero elements. More precisely, $\pi=(\pi_{i,j})_{i,j \geq 1} \in \Z_{ \geq 0}^{\Z_{>0} \times \Z_{>0}} $ is a plane partition if and only if 
\begin{align*}
\pi_{i+1,j} & \leq \pi_{i,j} \text{ and } \pi_{i,j+1} \leq \pi_{i,j} \text{ for all } i,j \in \Z_{>0}, \\
|\pi| & := \sum_{i,j \geq 1} \pi_{i,j} < +\infty.
\end{align*}

For $q \in (0,1)$, we consider the geometric probability measure $\tilde{\mathbb{P}}_q$ on the set of all plane partition given by
\begin{align*}
\tilde{\mathbb{P}}_q(\pi) = M q^{|\pi|},
\end{align*}
where $M$ is the normalization constant given by MacMahon formula (\cite{stanley}, corollary 7.20.3) 
\begin{align*}
M=\prod_{n=1}^{+\infty}(1-q^n)^n.
\end{align*}
To a plane partition we associate a subset of $E:=\Z\times \frac{1}{2}\Z$ via the map 
\begin{align*}
\pi \mapsto \mathfrak{S}_2(\pi):= \{ (i-j,\pi_{i,j}-(i+j-1)/2), \hspace{0.1cm} i,j \geq 1 \},
\end{align*}
see Figure 2.
The first coordinate of a point $(t,h) \in E$ might be interpreted as the time coordinate, and the second as the space coordinate.\\

The point process we consider, denoted by $\P_q$, is the image of $\tilde{\P}_q$ under $\mathfrak{S}_2$
\begin{align*}
\P_q(\mathcal{B}) = \tilde{\P}_q(\mathfrak{S}_2^{-1}(\mathcal{B})), \quad \mathcal{B} \subset \{0,1\}^E \text{ is a Borel set}.
\end{align*}
\begin{figure}
\begin{center}
\includegraphics[scale=0.6, trim = 1cm 1.3cm 1cm 0cm, clip=true]{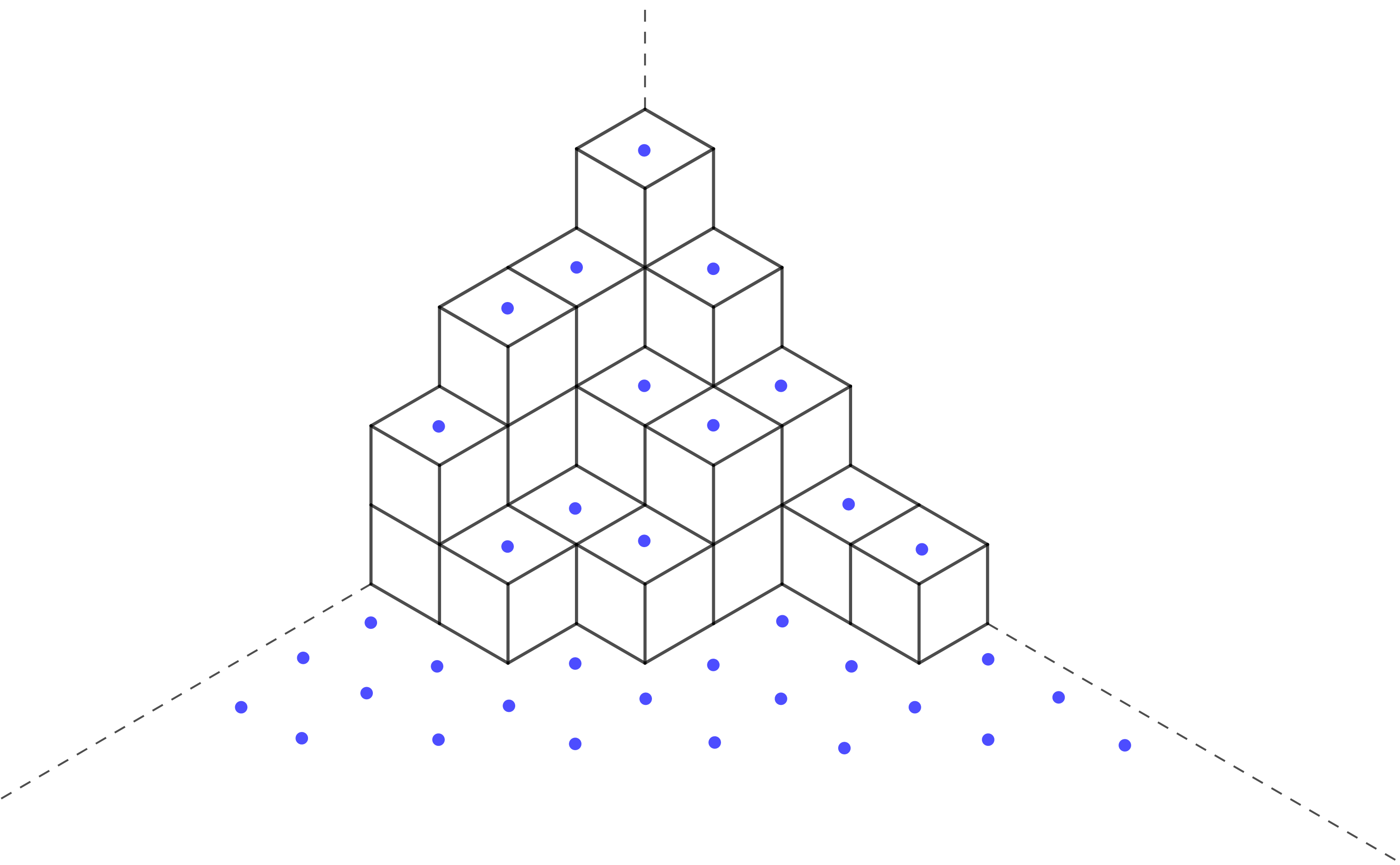}
\caption{The plane partition $\begin{pmatrix}
4 & 3 & 2 & 1 & 1 \\
3 & 2 & 2 & & \\
3 & 1 & 1 & & \\
2 & 1 & & &
\end{pmatrix}$ 
and its associated configuration.}
\end{center}
\end{figure}
\subsubsection{The determinantal formula}\label{sec:ppdpp}
Okounkov and Reshetikhin have shown in \cite{okounkovreshetikhin} that  $\mathbb{P}_q$ is a determinantal point procces on $E$. We define the kernel $K_q : E \times E \rightarrow \mathbb{R}$ by 
\begin{align} \label{defKq}
K_{q}(t_1,h_1;t_2,h_2)= \frac{1}{(2i\pi)^{2}}\int_{|z|=1 \pm \epsilon}\int_{|w|=1 \mp \epsilon} \frac{1}{z-w} \frac{\Phi(t_1,z)}{\Phi(t_2,w)}\frac{dz dw}{z^{h_1+\frac{|t_1|+1}{2}}w^{-h_2-\frac{|t_2|+1}{2}}}
\end{align}
where one picks the plus sign for $t_1 \geq t_2$ and the minus sign otherwise. The function $\Phi$ is defined by 
\begin{align*}
\Phi(t,z)= &\frac{(q^{1/2}/z;q)_\infty}{(q^{1/2+t}z;q)_\infty} \quad \text{for $t \geq 0$} \\
&\frac{(q^{1/2-t}/z;q)_\infty}{(q^{1/2}z;q)_\infty} \quad \text{for $t < 0$},
\end{align*}
where $(x;q)_\infty$ is a $q$ version of the Pochhammer symbol 
\begin{align*}
(x;q)_{\infty}=\prod_{k=0}^{\infty}(1-xq^{k}),
\end{align*}
and $\varepsilon$ is a sufficiently small positive number, which allows to avoid the singularities of the ratio 
\begin{align*}
\frac{\Phi(t_1,z)}{\Phi(t_2,w)}.
\end{align*}
Observe that there is no need of defining the square root in formula (\ref{defKq}), since, for $(t,h) \in E$, if there exists a plane partition $\pi$ such that $(t,h) \in \mathfrak{S}_2(\pi)$, then we have by construction that 
\begin{align*}
h+ \frac{|t|+1}{2} \in \Z.
\end{align*}
Okounkov-Reshetikhin determinantal formula is then the following statement 
\begin{thm}[Okounkov-Reshetikhin, \cite{okounkovreshetikhin}, 2003] \label{thmokresh} The point process $\P_q$ is a determinantal point process on $E$ with correlation kernel $K_q$. %for any pattern $m=\{ (t_1,h_1),...,(t_l,h_l) \} \subset E$, we have 
%\begin{align} \label{detformula}
%\tilde{\P} ( m \subset \mathfrak{S}(\pi) )=\mathbb{E}_q[c_m]=\det \left( K_q(t_i,h_i;t_j,h_j) \right)_{i,j=1}^l
%\end{align}
%where $\mathbb{E}_q$ denotes the expectation with respect to $\mathbb{P}_q$.
\end{thm}
The measure $\tilde{\P}_q$ is a particular case of a Schur process. Schur processes, first introduced in \cite{okounkovreshetikhin}, are dynamical generalizations of Schur measures in the sense that they form a Markov process on the set of partitions such that their marginal distributions are Schur measures. For a more elementary treatment of Schur processes, see e.g. \cite{borodinrains}, \cite{borodingorin} and references therein.
\subsubsection{The limit processes} \label{sec:limpp}
In the same article, Okounkov and Reshetikhin proved a scaling limit theorem for $\mathbb{P}_q$, when $r=-\log(q)$ tends to $0^+$, which we now formulate. Let us define 
\begin{align*}A:= \lbrace (t,x) \in \mathbb{R}^{2} , |2\cosh(t/2)-e^{-x}| < 2 \rbrace  \subset \mathbb{R}^2,
\end{align*}
and for $(\tau,\chi) \in A$, let $z(\tau,\chi)$ be the intersection point of the circles $C(0,e^{-\tau/2})$ and $C(1,e^{-\tau/4-\chi/2})$ with positive imaginary part, see Figure 3.
\begin{figure}
\includegraphics[scale=0.5,clip=true,trim=0cm 0cm 0cm 0cm]{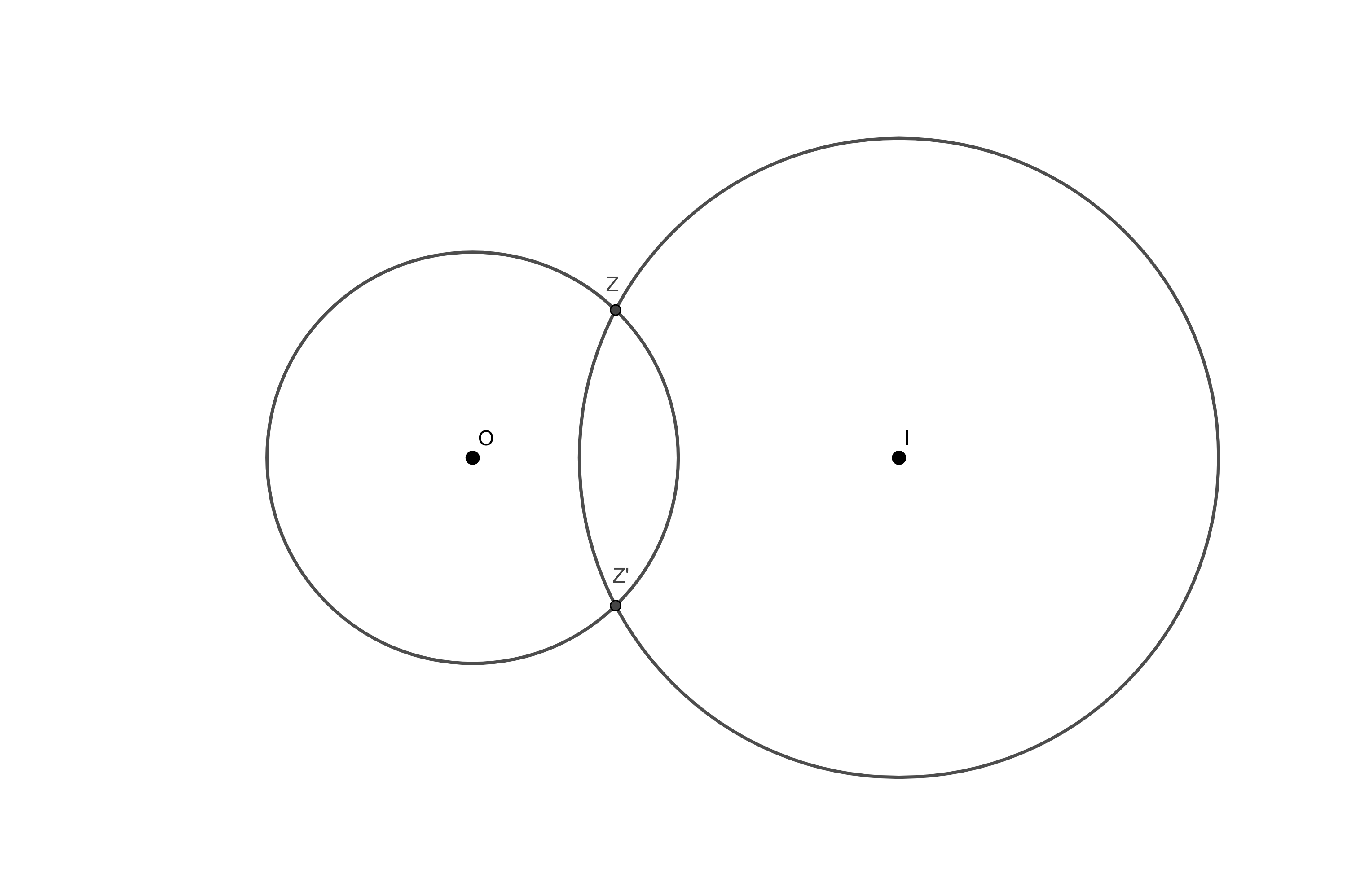}
\caption{The circles $C(0,e^{-\tau/2})$ and $C(1,e^{-\tau/4-\chi/2})$ and their intersection points $z=z(\tau,\chi)$ and its complex conjugate $z'=\overline{z(\tau,\chi)}$.}
\end{figure}
The condition $(\tau,\chi) \in A$ guarantees that $z(\tau,\chi)$ exists and is not real. For $(\tau,\chi) \in A$, we define the translation invariant kernel $\mathcal{S}_{z(\tau,\chi)} : E \rightarrow \mathbb{C}$
\begin{align} \label{defdynsin}
\mathcal{S}_{z(\tau,\chi)}(\Delta t, \Delta h)=\frac{1}{2i\pi}\int_{\overline{z(\tau,\chi)}}^{z(\tau,\chi)}(1-w)^{\Delta t}w^{-\Delta h - \frac{\Delta t}{2}} \frac{dw}{w},
\end{align}
where the integration path crosses $(0,1)$ for $\Delta t \geq 0$ and $(-\infty,0)$ for $\Delta t < 0$. For reasons explained below, this kernel will be called the \textit{extended sine kernel}. Then, the following holds
\begin{thm} [Okounkov-Reshetikhin, \cite{okounkovreshetikhin}] \label{thmcvokresh} For all $(\tau, \chi) \in A$ and all pattern $m=\{ (t_1,h_1),...,(t_l,h_l) \} \subset E$, we have 
\begin{align*}
\lim_{r\rightarrow 0}  \E_{e^{-r}} \left[c_{\frac{1}{r}(\tau,\chi)+m } \right]= \det \left( \mathcal{S}_{z(\tau,\chi)}(t_i-t_j,h_i-h_j) \right)_{i,j=1}^l.
\end{align*}
\end{thm}
For $(\tau, \chi) \in A$, the kernel $\mathcal{S}_{z(\tau,\chi)}$ defines a determinantal point process $\P_{(\tau,\chi)}$ on $E$ by
\begin{align} \label{sindynproc}
\forall m=\{ (m^1_1,m^2_1), ..., (m^1_l,m^2_l) \} \subset E, \quad \E_{(\tau,\chi)} [c_m] = \det\left( \mathcal{S}_{z(\tau,\chi)}(m_i^1-m_j^1,m_i^2-m_j^2)\right)_{i,j=1}^l,
\end{align}
where $\E_{(\tau,\chi)}$ is the expectation with respect to $\P_{(\tau,\chi)}$. This point process can be seen as a two-dimensional or dynamical version of the usual discrete sine-process on $\Z$ (see e.g. \cite{boo} or \cite{borodingorin}). Indeed, setting $\Delta t =0$ in (\ref{defdynsin}) leads to 
\begin{align*}
\mathcal{S}_{z(\tau,\chi)}(0,\Delta h) = e^{\frac{\tau \Delta h}{2}} \frac{\sin(\phi\Delta h)}{\pi\Delta h}
\end{align*}
where $z(\tau,\chi)= e^{-\frac{\tau}{2}+i\phi}$. Note that we can ignore the factor $e^{\frac{\tau \Delta h}{2}}$, see Remark \ref{rem:gaugetransform}.\\

We give the speed of convergence in Theorem \ref{thmcvokresh} above in the following Proposition.
\begin{prop} \label{prop:cvpp} For any compact $\mathcal{K} \subset \R^2$, and any pattern $m \subset E$, there exists $C>0$, such that for all $r>0$ sufficiently small and all $(\tau,\chi) \in A \cap \mathcal{K}$, one has 
\begin{align}
|\mathbb{E}_{e^{-r}}[c_{\frac{1}{r}(\tau,\chi)+m}]-\mathbb{E}_{(\tau,\chi)}[c_m]| \leq Cr.
\end{align}
\end{prop}
\begin{rem}We chose here and all along the the rest of the paper that concerns this model of random plane partitions to use the notation of \cite{okounkovreshetikhin}. In order to connect with the notation for the parameters used in the general framework of Section \ref{sec:intro}, we shall set
\[ \alpha := \frac{1}{r} = -\frac{1}{\log q} \]
and $u:=(\tau,\chi)$.
\end{rem}
The extended sine kernel appears in many other models as the kernel of the bulk scaling limit of two dimensional statistical mechanics models, for example non-intersecting paths (see e.g. \cite{gorinpaths} and references therein). It also has a continuous counter-part arising in the Dyson's brownian motion model (see e.g. \cite{katoritanemura}).

A similar model of random plane partitions has also been considered in \cite{ferrarispohn}, and related models of random skew plane partitions and their scaling limits have been studied for example in \cite{okounkovreshetikhin2}, \cite{boutilliermrktchyan}, \cite{mrktchyan} and \cite{mrktchyanpetrov}.
\subsubsection{The law of large numbers}\label{sec:llnpp}
For $r >0$, we define the set $A_{r} \subset E$ by 
\begin{align*}
A_{r}= r^{-1}A  \cap E.
\end{align*}
For brievety, we write $\P_r$ (resp. $\E_r$) instead of $\P_{e^{-r}}$ (resp. $\E_{e^{-r}}$). For a continuous compactly supported function $f : \R^2 \rightarrow \R$, and a pattern $m \subset E$, we form the empirical average
\begin{align} \label{sum}
\Sigma(f,m,r)=r^2 \sum_{(t,h) \in A_{r} }f(rt,rh)c_{(t,h)+m},
\end{align}
and we define the deterministic integral 
\begin{align*}
I(f,m) = \int_{A} f(\tau,\chi)\E_{(\tau,\chi)} [ c_m ] d\tau d\chi.
\end{align*}
Our Theorem establishes that, under $\P_r$, the sum $\Sigma(f,m,r)$ converges to $I(f,m)$.
\begin{thm} \label{thm:ppp}For every continuous compactly supported function $f: \mathbb{R}^2 \rightarrow \mathbb{R}$, every pattern $m \subset E$, one has 
\begin{align}
\forall \epsilon >0, \quad \lim_{r \rightarrow 0} \mathbb{P}_{r} \left( |\Sigma(f,m,r) - I(f,m) | > \epsilon \right) =0.
\end{align}
\end{thm}
\begin{rem}The assumption of compactness of the support of the function $f$ is used for the uniformity of constants in estimations of averages and variances. It might be interesting to see if Theorem \ref{thm:ppp} still holds for a wider class of functions, e.g. Schwartz functions.
\end{rem}

\subsection{Organization of the paper}
The paper is organized as follows. In Section \ref{sec:dpp}, we recall the definition of a determinantal point process on a discrete space $E$. We show in Proposition \ref{prop:covdet} a way to control covariances for disjoint patterns in a determinantal point process. The proposition will be used in the proofs of Lemma \ref{lem:varschur1} and especially in the proof of Lemma \ref{lemdecK}, as it provides a method of symmetrizing the problem posed by the fact that the correlation kernel $K_q$ is not symmetric.\\

Section \ref{sec:proofthmschur} is devoted to the proof of Theorem \ref{thm:schur}. We state preliminary results in Section \ref{sec:covschur} concerning the control of the covariances for the symmetric Schur measures, Lemmas \ref{lem:varschur1} and \ref{lem:varschur2}. We prove Theorem \ref{thm:schur} in Section \ref{sec:proofllnschur}, assuming Proposition \ref{prop:limschur} and Lemmas \ref{lem:varschur1} and \ref{lem:varschur2}. We prove Proposition \ref{prop:limschur} in Section \ref{sec:proofproplimschur}. In Section \ref{sec:rkshift}, we give an interpretation of Proposition \ref{prop:limschur} in terms of shift invariant subspaces of $l^2(\Z)$; this part is rather independant of the rest of the paper. We finally prove the decorrelation Lemmas \ref{lem:varschur1} and \ref{lem:varschur2} in Section \ref{sec:proofcovschur}.\\

We prove Theorem \ref{thm:ppp} in Section \ref{sec:proofthmpp}, the structure being the same as for the preceding section \ref{sec:proofthmschur}. We state Lemmas \ref{lemdecK} and \ref{lemdiag} concerning the control of the covariances in Section \ref{sec:covpp}, before proving Theorem \ref{thm:ppp} in Section \ref{sec:proofllnpp}, assuming Proposition \ref{prop:cvpp}, Lemmas \ref{lemdecK} and \ref{lemdiag}. We prove Proposition \ref{prop:cvpp} in Section \ref{sec:proofproplimpp}, and the decorrelation lemmas \ref{lemdecK} an \ref{lemdiag} in Section \ref{sec:proofcovpp}

\subsection*{Acknowledgements} I would like to thank Alexander Bufetov for posing the problem to me and for helpful discussions. I also would like to thank Alexander Boritchev, Nizar Demni and Pascal Hubert for helpful discussions and remarks.

This project has received funding from the European Research Council (ERC) under the European Union's Horizon 2020 research and innovation programme (grant agreement N 647133).

The author acknowledges the support of the fellowship "Assegni di ricerca FSE SISSA 2019"  from Fondo Sociale Europeo - Progetto SISSA OPERAZIONE 1 codice FP195673001.
\section{Determinantal point processes} \label{sec:dpp}
\subsection{Definitions}
We recall the following
\begin{defi}A point process $\P$ on $E$ is a \emph{determinantal point process} if there exists a kernel
\[K : E \times E \rightarrow \C, \]
called a correlation kernel for $\P$, such that, for any finite $m=\{m_1,\dots,m_l \} \subset E$, we have
\begin{align}\label{eq:defdpp}
\P ( m \subset X) = \E_\P [c_m] = \det \left(K(m_i,m_j) \right)_{i,j=1}^l.
\end{align}
\end{defi}
\begin{rem}\label{rem:gaugetransform} A correlation kernel of a given determinantal point process is not unique. Indeed, if $\P$ is a determinantal point process with correlation kernel $K$, then for any non-vanishing function $f : E \rightarrow \C$, the kernel
\[\tilde{K}(x,y) = \frac{f(x)}{f(y)} K(x,y) \]
also serves as a correlation kernel for $\P$, since the factors involving the function $f$ will disappear from any determinant of the form (\ref{eq:defdpp}).
\end{rem}
A correlation kernel $K$ defines to an integral operator on $l^2(E)$, which we denote by the same letter $K$
\begin{align*}
K :l^2(E) &\rightarrow l^2(E) \\
f &\mapsto \left( x \mapsto \sum_{y \in E} K(x,y)f(y) \right).
\end{align*}

Conversally, by the Macchi-Soshnikov/Shirai-Takahashi Theorem  (\cite{macchi}, \cite{soshnikov}, \cite{shiraitakahashi}), any locally trace class hermitian integral operator $K$ gives rise to a determinantal point process. In particular, any orthogonal projection onto a closed subspace $\mathcal{E}$ of $l^2(E)$ gives rise to a determinantal point process. 
\subsection{Covariances for discrete determinantal point processes}
We here provide a method for the estimation of covariances for discrete determinantal point processes. In particular, this method will be useful for determinantal point processes with a non-symmetric correlation kernel.
\begin{prop}\label{prop:covdet}Let $\P$ be a determinantal point process on $E$ with correlation kernel $K$. Let $m=\{ m_1,\dots, m_l\}$, $m'=\{m_1',\dots,m_{l'}' \} \subset E$ be two disjoint patterns. Then, the covariance
\begin{align} \label{eq:covgene}
\mathrm{cov}_\P(c_m,c_m'):= \E_\P [c_m c_{m'} ] - \E_\P [c_m] \E_\P[c_{m'}]
\end{align}
is a sum of $(l+l')! -l!l'!$ terms, each of one containing a factor of the form
\begin{align*}
K(m_i,m_j')K(m_i',m_j).
\end{align*}
\end{prop}
\begin{proof}Since the patterns $m$ and $m'$ are disjoint, we have 
\[c_mc_{m'}  =c_{m \sqcup m'}.\]
By definition of a determinantal point process, the expectation 
\[\E_\P [c_mc_{m'}]  \]
is an alternate sum of $(l+l')!$ indexed by the permutations of the set $m\sqcup m'$. Among these permutations, let us consider those that leave the sets $m$ and $m'$ invariant, i.e. permutations $\sigma$ which can be be factorized as $\sigma=\tau \tau'$ where $\tau$ (resp. $\tau'$) is a permutation of $m$ (resp. $m'$). The alternate sum over all the permutations leaving the sets $m$ and $m'$ invariant is nothing but the product of the determinants
\[ \E_\P[c_m] \E_\P [c_{m'}]. \]
Thus, the remaining terms in the covariance (\ref{eq:covgene}) are all indexed by permutations $\sigma$ leaving neither $m$ nor $m'$ invariant (since $\sigma $ is a permutation, if $\sigma$ left $m$ invariant, it would also leave $m'$ invariant), i.e. there exist $m_i,m_j \in \{m_1,\dots,m_l\}$ and $m_i',m_j' \in \{m_1',\dots,m_{l'}'\}$ such that
\begin{align*}
\sigma(m_i)=m_j' \quad \text{and} \quad \sigma(m_{i}')=m_j.
\end{align*}
The proof is complete.
\end{proof}
\section{Proof of Theorem \ref{thm:schur}} \label{sec:proofthmschur}

\subsection{Control of the covariances} \label{sec:covschur}
Let $\alpha >0$ and let $\P_\alpha$ be the image by $\mathfrak{S}$ of the symmetric Schur measure $\mathbb{S}_\alpha$ defined in section \ref{sec:introllnschur}. We start by stating preliminary results on the control of the covariances
\begin{align*}
\mathrm{cov}_\alpha (c_{m + \alpha u_1},c_{m + \alpha u_2})= \E_\alpha [c_{m+\alpha u_1} c_{m + \alpha u_2} ] - \E_\alpha [c_{m+\alpha u_1}] \E_\alpha [ c_{m + \alpha u_2} ],
\end{align*}
distinguishing the cases when $u_1$, $u_2 \in [u_{\min}, u_{\max}]$ are away from each other or not.
\begin{lem}\label{lem:varschur1}Let $m = \{m_1,\dots,m_l\} \subset \Z$ be a pattern and denote by $\overline{m}$ its supremum norm
\[\overline{m} = \max \{|m_i|, \hspace{0.1cm} i=1,\dots,l\}. \]
There exists $C>0$ such that for all sufficiently large $\alpha>0$ and for any $u_1,u_2 \in [u_{\min},u_{\max}]$ such that
\[|u_1-u_2| > \frac{\overline{m}}{\alpha}, \]
we have
\begin{align}
 \left| \mathrm{cov}_\alpha(c_{m+ \alpha u_1 }, c_{m +  \alpha u_2 }) \right| \leq \frac{C}{\alpha |u_1-u_2|}.
\end{align}
\end{lem}
\begin{lem}\label{lem:varschur2}For any patterns $m,m' \subset \Z$, there exists $C>0$ such that for any $u \in [u_{\min},u_{\max}]$ and any sufficiently large $\alpha >0$ we have
\begin{align}
\left| \mathrm{cov}_\alpha( c_{m + \alpha u}, c_{m' + \alpha u}) \right| \leq C.
\end{align}
\end{lem}
\subsection{Proof of Theorem \ref{thm:schur}} \label{sec:proofllnschur}
For brevety, we write
\begin{align*}
A_\alpha:= [\alpha u_{\min}, \alpha u_{\max}] \cap \Z.
\end{align*}
We start by writing
\begin{align*}
\E_\alpha \Sigma(f,m,\alpha)= \frac{1}{\alpha} \sum_{x \in A_\alpha} f \left(\frac{x}{\alpha} \right) \E_{\alpha} [ c_{m + x} ] = \frac{1}{ \alpha} \sum_{ u \in [ u_{\min}  ,  u_{\max}] \cap \frac{1}{\alpha}\Z} f(u) \E_{\alpha} [c_{m + \alpha u} ]. 
\end{align*}
By Proposition \ref{prop:limschur}, and since
\begin{align*}
\left| A_\alpha \right| = O \left(\alpha \right),
\end{align*}
we have that, for any $\delta \in (0,1)$, there exists $C>0$ such that
\begin{align} \label{ineq:cvexpschur}
 \left|\E_\alpha \Sigma(f,m,\alpha) - \frac{1}{\alpha} \sum_{ u \in [ u_{\min}  ,  u_{\max}] \cap \frac{1}{\alpha}\Z} f(u) \E_{\mathcal{S}(u)} [c_m] \right| 
\leq C \alpha  \exp \left(-\alpha^{-\delta} \right).
\end{align}
The sum
\begin{align*}
\frac{1}{\alpha} \sum_{ u \in [ u_{\min}  ,  u_{\max}] \cap \frac{1}{\alpha}\Z} f(u) \E_{\mathcal{S}(u)} [c_m] = \frac{1}{\alpha} \sum_{ x \in A_\alpha} f \left( \frac{x}{\alpha} \right) \E_{\mathcal{S}(x/\alpha)} [c_m]
\end{align*}
is a Riemann sum for the integral
\begin{align*}
I(f,m)= \int_{u_{\min}}^{u_{\max}} f(u) \E_{\mathcal{S}(u)} [c_m] du,
\end{align*}
whence we deduce from (\ref{ineq:cvexpschur}) that
\begin{align*}
\lim_{\alpha \rightarrow + \infty} \E_\alpha \Sigma(f,m ,\alpha) = I(f,m).
\end{align*}
It suffices now to prove that the variance
\begin{multline} \label{eq:varschur}
\mathrm{var}_\alpha \Sigma(f,m,\alpha) = \E_\alpha \left| \Sigma(f,m,\alpha) - \E_\alpha \Sigma(f,m,\alpha) \right|^2 \\
= \frac{1}{\alpha^2} \sum_{x,y \in A_\alpha} f\left( \frac{x}{\alpha} \right) f \left( \frac{y}{\alpha} \right) \mathrm{cov}_{\alpha} (c_{m + x}, c_{m + y})
\end{multline}
goes to $0$ as $\alpha \rightarrow + \infty$. We cut the set $A_\alpha^2$
into two parts by setting
\begin{align*}
A_\alpha^>&= \{(x,y) \in A_\alpha^2, \hspace{0.1cm} |x-y| > \overline{m} \} , \\
A_\alpha^\leq &= A_\alpha^2 \setminus A_\alpha^>,
\end{align*}
and decompose the variance (\ref{eq:varschur}) into two sums
\begin{multline} \label{eq;varschurcut}
\mathrm{var}_\alpha \Sigma(f,m,\alpha)  = \frac{1}{\alpha^2} \left( \sum_{(x,y) \in A_\alpha^>} f\left( \frac{x}{\alpha} \right) f \left( \frac{y}{\alpha} \right) \mathrm{cov}_{\alpha} (c_{m + x}, c_{m + y}) \right.\\
+ \left. \sum_{ (x,y) \in A_\alpha^\leq} f\left( \frac{x}{\alpha} \right) f \left( \frac{y}{\alpha} \right) \mathrm{cov}_{\alpha} (c_{m + x}, c_{m + y}) \right).
\end{multline}
By construction, for all $(x,y) \in A_\alpha^>$, there exists $u_1,u_2 \in [u_{\min},u_{\max}] \cap \frac{1}{\alpha }\Z$ satisfying
\[|u_1-u_2| > \frac{\overline{m}}{\alpha} \]
and such that
\begin{align*}
x=\alpha u_1, \quad y = \alpha u_2.
\end{align*}
By lemma \ref{lem:varschur1}, there exists $C >0$ such that for all $(x,y) \in A_\alpha^>$, we have
\begin{align} \label{ineq:covschur1}
|\mathrm{cov}_{\alpha}(c_{m+x},c_{m+ y })| \leq \frac{C}{|x-y|}.
\end{align}
Since
\[|A_\alpha^>| =O \left( \alpha^2\right), \]
and since the function $f$ is bounded on $[u_{\min},u_{\max}]$, we deduce from (\ref{ineq:covschur1}) that
\begin{align} \label{ineq:varchur1}
\sum_{(x,y) \in A_\alpha^>}  \left| f\left( \frac{x}{\alpha} \right) f \left( \frac{y}{\alpha} \right)  \mathrm{cov}_{\alpha} (c_{m + x}, c_{m + y}) \right| \leq C' \sum_{(x,y) \in A_\alpha^>} \frac{1}{|x-y|}  = O \left( \alpha \log(\alpha) \right).
\end{align}
Now, for any $(x,y) \in A_\alpha^\leq$, there exist $u \in [u_{\min}, u_{\max}]$ and patterns $m', m'' \subset \Z$ such that
\begin{align*}
x+m = \alpha u + m', \quad y+ m = \alpha u + m''.
\end{align*}
Observe that there is only a finite number of possibilities for the patterns $m', m''$ as $(x,y)$ ranges over $A_\alpha^\leq$. Thus by Lemma \ref{lem:varschur2}, there exists $C>0$ such that for all $(x,y) \in A_\alpha^\leq$, we have
\begin{align} \label{ineq:covschur2}
\left| \mathrm{cov}_\alpha(c_{m+x,m+y}) \right| \leq C.
\end{align}
We obtain from (\ref{ineq:covschur2}) that there exists $C'>0$ such that 
\begin{align} \label{ineq:varschur2}
\sum_{(x,y) \in A_\alpha^\leq} \left| f\left( \frac{x}{\alpha} \right) f \left( \frac{y}{\alpha} \right)  \mathrm{cov}_{\alpha} (c_{m + x}, c_{m + y}) \right| \leq C'\alpha.
\end{align}
Inserting inequalities (\ref{ineq:varchur1}) and (\ref{ineq:varschur2}) into (\ref{eq;varschurcut}), we obtain
\begin{align*}
\mathrm{var}_\alpha \Sigma(f,m,\alpha) = O \left( \frac{\log \alpha}{\alpha } \right) 
\end{align*}
and the proof is complete.

\subsection{Proof of Proposition \ref{prop:limschur}} \label{sec:proofproplimschur}
For $u \in [u_{\min},u_{\max}]$, we define the function
\begin{align*}
S_u(z) = G(z) - u\log(z),
\end{align*}
for $z$ in a neighborhood of the unit circle $\T$, except at the intersection with the semi-axis $(-\infty,0]$, so that the logarithm is well defined. Let $K_\alpha$ be the correlation kernel of the point process $\P_\alpha$. We will prove that
\begin{align} \label{lim:schur}
\lim_{ \alpha \to + \infty} K_\alpha( \alpha u + x , \alpha u + y) = \mathcal{S}(u)(x,y),
\end{align}
with a remaining term of order at most $\exp \left( -\alpha^\delta \right)$, $\delta \in (0,1)$. By Theorem \ref{thm:okounkovschurdet}, we have
\begin{align} \label{eq:intkernasympschur}
K_\alpha(x +\alpha u, y+ \alpha u) = \frac{1}{(2i \pi)^2} \int_{|z|=1 + \varepsilon} \int_{|w| = 1 - \varepsilon} \frac{\exp \left(\alpha \left(S_u(z) - S_u(w) \right)\right)}{z-w} \frac{dzdw}{z^{x+1}w^{-y+1}}.
\end{align}
We want to deform the contours of integration in order to have $\Re S_u(z) - \Re S_u(w) <0$. To this aim, observe first that the real part of the function $S_u$ is constant and equals zero on the unit circle $\T$. Identifying the complex plane with $\R^2$, the gradient of the real part of the function $S_u$ is thus orthogonal to $\T$. Its direction is given by the sign of the scalar product
\begin{align*}
\langle \nabla \Re S_u(z) , z \rangle_{\R^2} = zG'(z) -u.
\end{align*}
Recall that for $u \in [u_{\min},u_{\max}]$, the set
$\{ z \in \T, \hspace{0.1cm} zG'(z)\geq u \}$ consists of $L=L(u)$ arcs $[e^{i\phi_k(u)}e^{i\psi_k(u)} ]$, $k=1,\dots ,L(u)$, and observe that the numbers $e^{i\phi_k(u)}$, $e^{i\psi_k(u)}$ are the critical points of the function $\Re S_u$. One deforms the circle $\T$ into a contour $\gamma_u^>$ by following the direction of the gradient $\nabla \Re S_u$, and into an other contour $\gamma_u^{<}$ by following the opposite direction of the gradient$\nabla \Re S_u$, see Figure 4.
\begin{figure} \label{fig:contoursdefschurlim}
\includegraphics[scale=0.2, trim = 0cm 5cm 0cm 4cm, clip =true]{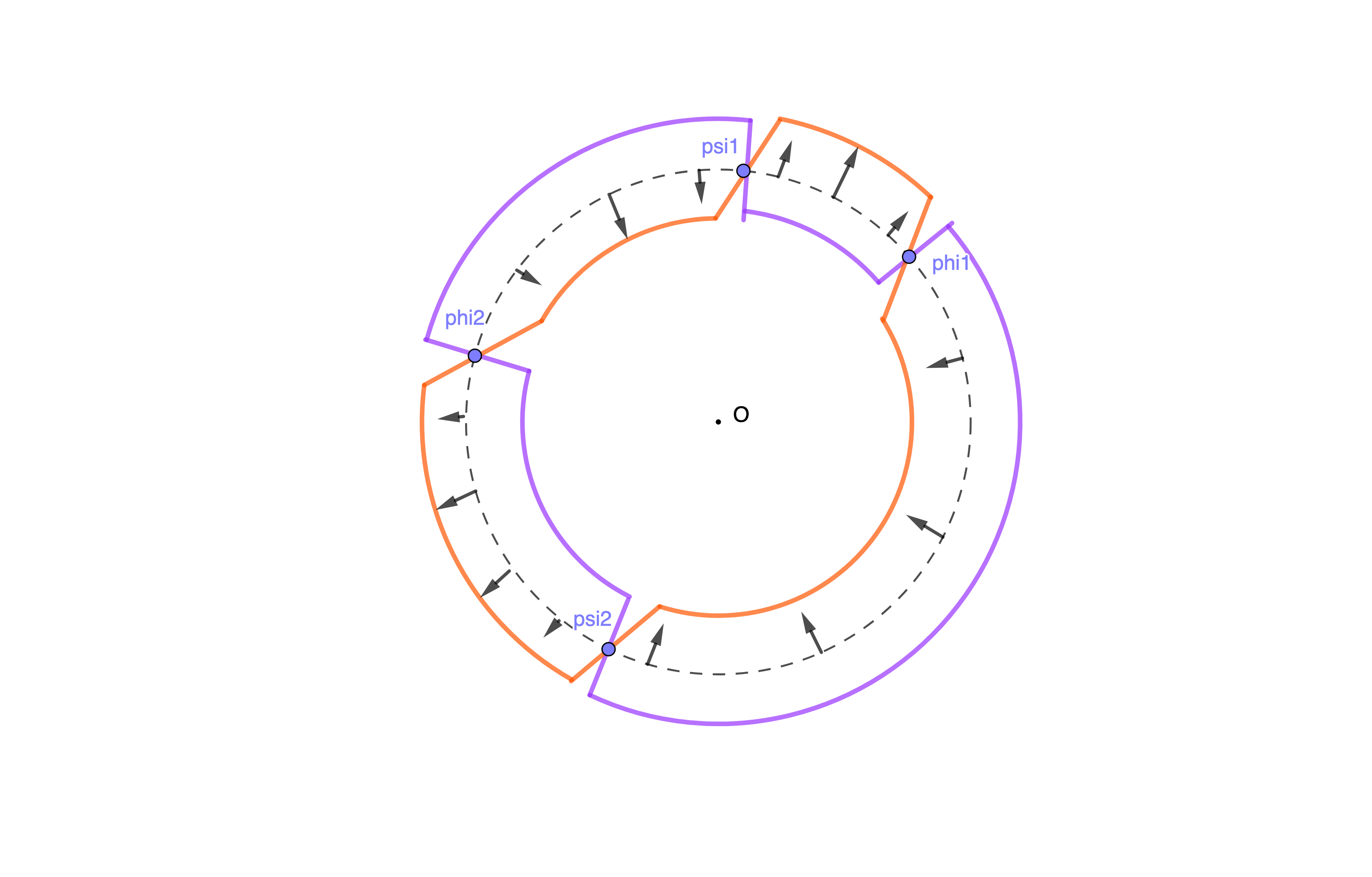}
\caption{The variable $z$ belongs to the purple contour $\gamma_u^<$, while $w$ belongs to the orange one $\gamma_u^>$.}
\end{figure}
By construction, we have
\begin{align*}
\Re S_u(z)  \geq 0, \quad z \in \gamma_u^>,
\end{align*}
and
\begin{align*}
\Re S_u(z) \leq 0, \quad z \in \gamma_u^<,
\end{align*}
with equality only at the critical points $e^{i\phi_k(u)}$, $e^{i\psi_k(u)}$, $k=1,\dots,L(u)$. By the preceding discussion, the dominated convergence Theorem implies that
\begin{align} \label{eq:intwithoutresidueschur}
\lim_{\alpha \rightarrow + \infty} \exp \left( \alpha^\delta \right)\int_{z \in \gamma_u^<} \int_{w \in \gamma_u^>} \frac{\exp \left( \alpha \left(S_u(z)-S_u(w)\right) \right)}{z-w} \frac{dzdw}{z^{x+1}w^{-y+1}} =0
\end{align}
for all $u \in [u_{\min},u_{\max}]$, all $\delta \in (0,1)$ and all $x,y \in \Z$. Observe now that the residue at $z=w$ integrated in the expression (\ref{eq:intkernasympschur}) is no longer integrated in (\ref{eq:intwithoutresidueschur}) on the arcs $[e^{i\phi_k(u)},e^{i\psi_k(u)}]$. Since this residue is $w^{y-x-2}$, we have from (\ref{eq:intwithoutresidueschur})
\begin{multline*}
K_\alpha(x+ \alpha u, y + \alpha u) \\
= \frac{1}{(2i\pi)^2}\int_{z \in \gamma_u^<} \int_{w \in \gamma_u^>} \frac{\exp \left( \alpha \left(S_u(z)-S_u(w)\right) \right)}{z-w} \frac{dzdw}{z^{x+1}w^{-y+1}}
 + \sum_{k=1}^{L(u)}\frac{1}{2\pi} \int_{\phi_k(u)}^{\psi_k(u)} e^{i\theta(y-x-1)}d\theta \\
=O\left( \exp \left( -\alpha^\delta \right) \right) + \mathcal{S}(u)(x,y).
\end{multline*}
Since $[u_{\min},u_{\max}]$ is compact, the $O\left( \exp \left( -\alpha^\delta \right) \right)$ is uniform in $u$. Recalling Definition (\ref{eq:defdpp}), we obtain (\ref{limprop:limschur}) and the proof is complete.
\subsection{A remark on shift invariant subspaces} \label{sec:rkshift}
\subsubsection{Wiener-Helson's Theorem on the classification of shift invariant spaces}
We here give a classification of shift invariant subspaces of $l^2(\Z)$ due to Wiener and Helson, Theorem \ref{thm:wienerhelson} below. The material we present is quite standard, and may be found for instance in the first pages of the book \cite{nikolski1}.
\\

Let $S : l^2(\Z) \to l^2(\Z)$ be the shift
\begin{align*}
S (a_k)_{k \in \Z} = (a_{k-1})_{k \in \Z}, \quad (a_k)_{k \in \Z} \in l^2(\Z).
\end{align*}
We say that a closed subspace $\mathcal{E} \subset l^2(\Z)$ is shift invariant if $S(\mathcal{E}) \subset \mathcal{E}$. If $\mathcal{E}$ is shift invariant, it is said to be \emph{simply} invariant if $S(\mathcal{E}) \neq \mathcal{E}$ and doubly invariant otherwise. \\

We equip the unit circle $\T$ with the normalized Lebesgue measure and recall the Fourier-Plancherel unitary isomorphism
\begin{align*}
\mathcal{F} : L^2(\T) &\to l^2(\Z) \\
F &\mapsto (\hat{F}(k))_{k \in \Z},
\end{align*}
where
\begin{align*}
\hat{F}(k):= \frac{1}{2\pi}\int_{-\pi}^\pi e^{-i\theta k} F(e^{i\theta}) d\theta = \frac{1}{2i \pi} \int_{|z|=1} z^{-k} F(z) \frac{dz}{z}.
\end{align*}
The map $S$ is conjugated to the multiplication by $z$ by the Fourier-Plancherel isomorphism.\\

The Hardy space $H^2(\T)$ is defined as the closed linear space of functions of $L^2(\T)$ with vanishing negative Fourier coefficients:  $H^2(\T):= \{ F \in L^2(\T), \hspace{0.1cm} \hat{F}(k)=0, \hspace{0.1cm} k=-1,-2,\dots \}$. Before stating the Wiener-Helson's Theorem classifying the subspaces of $l^2(\Z)$ which are invariant by the shift $S$ (or, equivalently, the subspaces of $L^2(\T)$ which are invariant by the multiplication by $z$), we make the following elementary observation: if a function $F \in L^2(\T)$ is such that $|F(z)|=1$ for almost all $z \in \T$, then the functions $z \mapsto F(z)z^n$, $n=0,1,\dots$ form an orthonormal basis of $FH^2(\T)$, and the space $FH^2(\T)$ is invariant by the multiplication by $z$. Let us write
\begin{align*}
\mathcal{E}_F:= \mathcal{F}(F H^2(\T)).
\end{align*}
The space $\mathcal{E}_F$ is then a simply shift invariant subspace of $l^2(\Z)$. Helson's Theorem, which is the second point of the Theorem \ref{thm:wienerhelson} below, establishes the reversed statement. The first point is due to Wiener, and classifies doubly shift invariant subspaces.
\begin{thm} \label{thm:wienerhelson}
Let $\mathcal{E} \subset l^2(\Z)$ be a shift invariant subspace. Then we have the following
\begin{enumerate}
\item[$\bullet$] If the subspace $\mathcal{E}$ is doubly invariant, then there exists a borel set $A \subset \T$ such that $\mathcal{E}= \mathcal{F} \left( \mathfrak{1}_A  L^2(\T) \right) $, where $\mathfrak{1}_A$ is the indicator function of $A$.
\item[$\bullet$] If the subspace $\mathcal{E}$ is simply invaraint, then there exists a function $F \in L^2(\T)$ verifying $|F(z)|=1$ for almost every $z \in \T$ and such that $\mathcal{E}=\mathcal{F} \left( F H^2(\T) \right)$. 
\end{enumerate}
\end{thm}
\subsubsection{The Schur measures revisited} \label{sec:schurmeasuresrevisited}
We now connect the preceding discussion with Schur measures and their limit processes. Let $G$ be a function holomorphic in a neighborhood of the unit circle $\T$ such that $G(z) \in i \R$ for $z \in \T$, i.e. such that $\hat{G}(-k)=- \overline{ \hat{G}(k)}$ for all $k \in \Z$. We set 
\begin{align*}
F(z) = \exp \left( G(z^{-1}) \right).
\end{align*}
We have $|F(z)|=1$ for almost every $z \in \T$.

Let $K_F$ be the kernel of the orthogonal projection onto $\mathcal{E}_F$. By the Macchi-Soshnikov/Shirai-Takahashi Theorem, the kernel $K_F$ serves as a correlation kernel of a determinantal point process. The following Proposition says that the determinantal point process with kernel $K_F$ is a symmetric Schur measure, up to the transformation $X \mapsto -X$, $X \in \Conf(\Z)$.
\begin{prop} \label{prop:shiftschur} Let $\rho : \Lambda \to \C$ be the specialization defined by 
\begin{align*}
\rho(p_k)= k \hat{G}(k), \quad k=1,2,\dots
\end{align*}
and let $\P_\rho$ be the symmetric Schur measure with specialization $\rho$. Let $K_\rho$ be the correlation kernel of the corresponding determinantal point process $\mathfrak{S}^* \P_\rho$ on $\Z$. Then, for all $x, y \in \Z$, we have
\begin{align}
K_F(x,y) = K_\rho(-x,-y).
\end{align}
\end{prop}
\begin{proof}
Since the family $\{ (\hat{F}(k-n))_{k \in \Z}, \hspace{0.1cm} n=0,1,\dots \}$ is an orthonormal basis of $\mathcal{E}_F$, we have
\begin{align*}
K_F(x,y) &= \sum_{n =0}^{+ \infty} \hat{F}(x-n) \overline{\hat{F}(y-n)} \\
&= \sum_{n=0}^{+ \infty} \frac{1}{(2  \pi)^2} \int_{-\pi}^\pi e^{-i\theta(x-n)} F(e^{i\theta}) d\theta \int_{-\pi}^\pi e^{i\theta'(y-n)} F^{-1}(e^{i\theta'})d\theta',
\end{align*}
where we used the fact that $|F(z)|=1$ for $z \in \T$. Changing the variable $\theta \mapsto -\theta$, $\theta' \mapsto -\theta'$, we obtain
\begin{align*}
K_F(x,y) &= \sum_{n = 0}^{+ \infty} \frac{1}{(2\pi)^2} \int_{-\pi}^\pi e^{i \theta(x-n)} F(e^{-i \theta}) d\theta \int_{- \pi}^\pi e^{i\theta'(n-y)} F^{-1}(e^{-i \theta'}) d\theta' \\
&= \sum_{n = 0}^{+ \infty} \frac{1}{(2i \pi)^2} \int_{|z|=1} z^{x-n} F(z^{-1})  \frac{dz}{z} \int_{|w|=1} w^{n-y} F^{-1}(w^{-1}) \frac{dw}{w}.
\end{align*}
Since $F$ is holomorphic in a neighborhood of $\T$, we have for $\varepsilon >0$ small enough that
\begin{align*}
K_F(x,y) &=  \sum_{n = 0}^{+ \infty} \frac{1}{(2i \pi)^2} \int_{|z|=1+ \varepsilon} z^{x-n} F(z^{-1})  \frac{dz}{z} \int_{|w|=1- \varepsilon} w^{n-y} F^{-1}(w^{-1}) \frac{dw}{w} \\
&= \frac{1}{(2 i \pi)^2} \int_{|z|=1+ \varepsilon} \int_{|w|=1-\varepsilon} \frac{F(z^{-1})F^{-1}(w^{-1})}{1-w/z} \frac{dzdw}{z^{-x+1}w^{y+1}} \\
&=\frac{1}{(2 i \pi)^2} \int_{|z|=1+ \varepsilon} \int_{|w|=1-\varepsilon} \frac{F(z^{-1})F^{-1}(w^{-1})}{z-w} \frac{dzdw}{z^{-x}w^{y+1}},
\end{align*}
where we used Fubini's Theorem for the second line. Observing that from the definitions of the function $F$ and of the specialization $\rho$ we have
\begin{align*}
F(z^{-1})F^{-1}(w^{-1})&=\exp \left(G(z)-G(w)) \right) \\
&=\exp \left( \sum_{k \geq 1} \hat{G}(k)(z^{k}  -w^k) +\overline{\hat{G}(k)}(w^{-k}-z^{-k} ) \right)  \\
&= \exp \left( \sum_{k \geq 1} \frac{\rho(p_k)}{k}(z^k-w^{k}) + \frac{\overline{\rho(p_k)}}{k}(w^{-k} - z^{-k}) \right),
\end{align*}
and recalling Formula (\ref{eq:kernelschur}) for the kernel $K_\rho$, the proof is complete.
\end{proof}
\begin{rem}As observed in \cite{bufetovolshanski}, Proposition \ref{prop:shiftschur} above remains true for more general symmetric Schur measures, i.e. when the function $G$ is not necessarly holomorphic in a neighborhood of $\T$. We chose to restrict ourselves to the cas when $G$ is holomorphic in a neighborhood of $\T$ in order to enlight the interpretation of Proposition \ref{prop:limschur} in terms of shift invariant subspaces, which we describe below.
\end{rem}
\subsubsection{An interpretation of Proposition \ref{prop:limschur}}
Observe that the limit kernels $\mathcal{S}(u)$ are projection kernels  onto doubly shift invariant subspaces 
\[\mathcal{E}(u):= \mathcal{F} \left( \mathfrak{1}_{\sqcup [e^{i\phi_k(u)},e^{i\psi_k(u)}] } L^2(\T) \right),\] since we have
\begin{align*}
\mathcal{S}(u)(x,y)= \sum_{k=1}^{L(u)} \hat{\mathfrak{1}}_{[e^{i\phi_k(u)},e^{i\psi_k(u)}]}(x-y).
\end{align*}

Proposition \ref{prop:limschur}, and more precisely the convergence (\ref{lim:schur}) establishes thus a link between simply and doubly shift invariant subspaces of $l^2(\Z)$. In a certain regime and in a certain sense, some simply invariant subspaces converge to doubly invariant subspaces. We make this fact precise: if $G$ is a holomorphic function in the neighborhood of $\T$ as above, we can define a function $F_\alpha(z)= \exp \left(\alpha G(z^{-1}) \right)$, and consider the simply invariant subspace $\mathcal{E}_{F_\alpha}$. With a little computation, we have that the kernel 
\[(x,y) \mapsto K_{F_\alpha}(\lfloor \alpha u \rfloor + x, \lfloor \alpha u \rfloor +y)\]
is the kernel of the orthogonal projection onto the subspace $S^{-\lfloor \alpha u \rfloor} (\mathcal{E}_{F_\alpha})$. By Proposition \ref{prop:shiftschur} and slighly adapting the proof of the convergence (\ref{lim:schur}), we have that, as $\alpha \to + \infty$, the space $S^{-\lfloor \alpha u \rfloor} (\mathcal{E}_{F_\alpha})$ converges to $\mathcal{E}(u)$ in the sense that the orthogonal projection onto $S^{-\lfloor \alpha u \rfloor } (\mathcal{E}_{F_\alpha})$ converges to the orthogonal projection onto $\mathcal{E}(u)$ in the strong operator topology.
\subsection{Proof of Lemmas \ref{lem:varschur1} and \ref{lem:varschur2}} \label{sec:proofcovschur}
\subsubsection{Proof of Lemma \ref{lem:varschur1}}
By Proposition \ref{prop:covdet} and since the kernel $K_\alpha$ is symmetric, it suffices to prove that for all $x,y \in \Z$, there exists $C>0$ such that
\begin{align} \label{ineq:covschurkernel}
|K_\alpha(x+\alpha u_1,y+ \alpha u_2) | \leq \frac{C}{\alpha | u_1-u_2|}
\end{align}
for all $u_1,u_2 \in [u_{\min},u_{\max}]$, $u_1 \neq u_2$ and all $\alpha>0$ sufficiently large. As in the proof of Proposition \ref{prop:limschur}, write
\begin{align*}
K_\alpha(x+ \alpha u_1, y+ \alpha u_2) = \frac{1}{(2i \pi)^2} \int_{|z|=1 + \varepsilon} \int_{|w|=1- \varepsilon} \frac{\exp \left(S_{u_1}(z) - S_{u_2}(w) \right) }{z-w} \frac{dzdw}{z^{x+1}w^{-y+1}}.
\end{align*}
We then deform the contour and integrate over $z \in \gamma_{u_1}^<$ and $w \in \gamma_{u_2}^>$ so that
\begin{align*}
\Re S_{u_1}(z) - \Re S_{u_2}(w) <0,
\end{align*}
except at a finite number of points, see Figure 5.
\begin{figure} \label{fig:covschur}
\includegraphics[scale=0.2, trim = 0cm 4cm 0cm 4cm, clip =true]{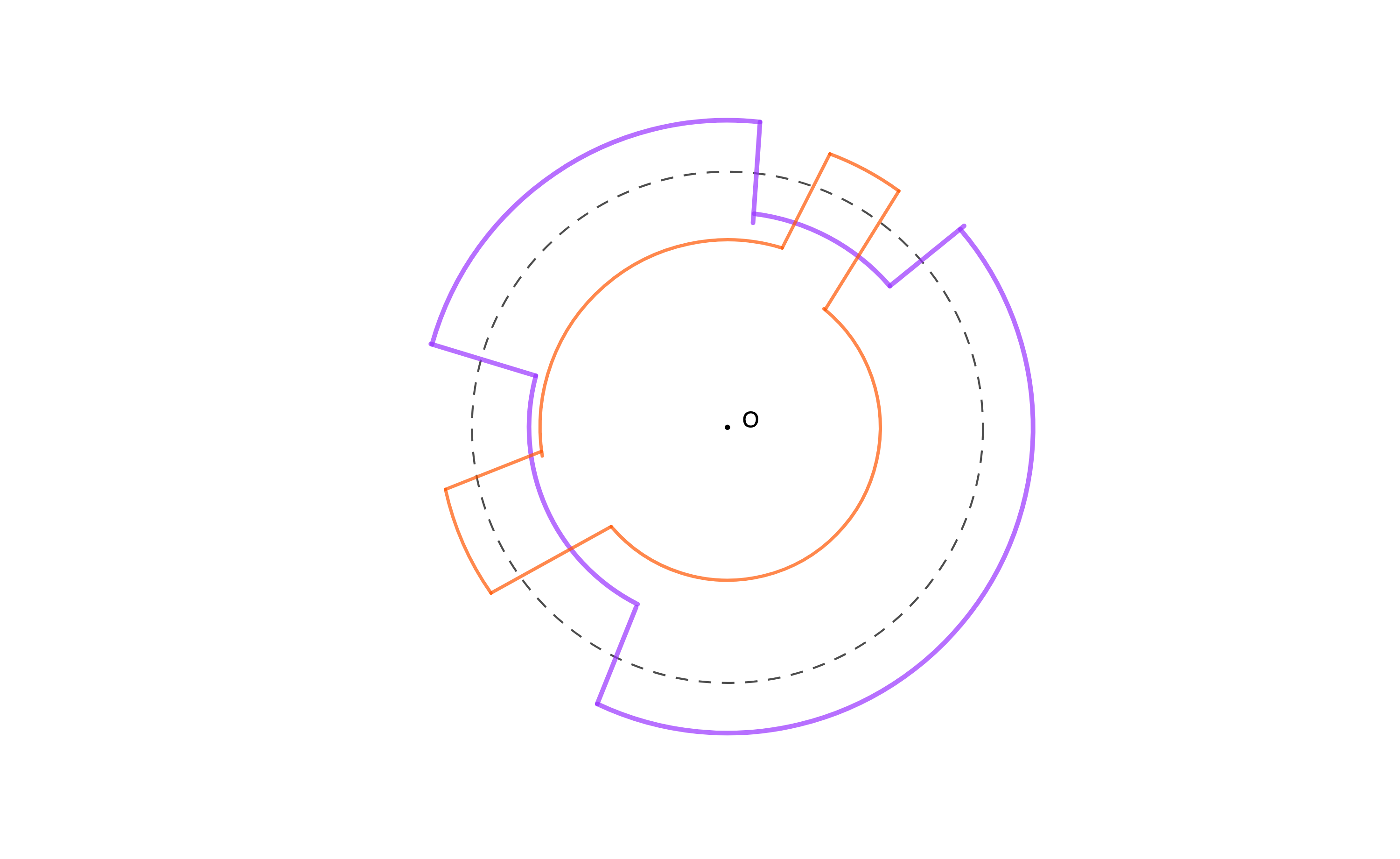}
\caption{The contours $\gamma_{u_1}^<$ and $\gamma_{u_2}^>$.}
\end{figure}In order to recover the value of the kernel $K_\alpha$, we must again integrate the residue at $z=w$ over a finite number of arcs $[e^{i\phi_k},e^{i\psi_k}]$ which depend on $u_1$ and $u_2$. We have
\begin{align*}
K_\alpha(x+ \alpha u_1, y+ \alpha u_2)  = \frac{1}{(2i \pi)^2} \int_{ z \in \gamma_{u_1}^<} \int_{w \in \gamma_{u_2}^>} \cdots + \sum_{k} \frac{1}{2 \pi} \int_{\phi_k}^{\psi_k} e^{i \theta (\alpha u_2 - \alpha u_1 + y -x  - 1)} d\theta.
\end{align*}
By construction, the integral
\begin{align*}
\int_{ z \in \gamma_{u_1}^<} \int_{w \in \gamma_{u_2}^>} \cdots
\end{align*}
is dominated by $\exp(-\alpha^\delta)$, for any $\delta \in (0,1)$. By a direct integration, we have
\begin{align*}
 \left| \int_{\phi_k}^{\psi_k} e^{i \theta (\alpha u_2 - \alpha u_1 + y -x  - 1)} d\theta \right| \leq \frac{C}{\alpha |u_1-u_2|}
\end{align*}
for all $\alpha>0$ sufficiently large and all $u_1 \neq u_2$, where $C$ only depens on $x$ and $y$. We have just established (\ref{ineq:covschurkernel}), and the lemma is proved.
\subsubsection{Proof of Lemma \ref{lem:varschur2}}
Since the points $e^{i\phi_k(u)}$, $e^{i\psi_k(u)}$ are the solutions of the equation $zG'(z)=u$, the function 
\[u \mapsto \mathcal{S}(u)(x,y) \]
is continuous in $u$ for any fixed $x,y \in \Z$. Since $[u_{\min},u_{\max}]$ is compact, this established the proof of Lemma \ref{lem:varschur2}.

\section{Proof of Theorem \ref{thm:ppp}} \label{sec:proofthmpp}
\subsection{Control of the covariances} \label{sec:covpp}
Let $r>0$ and let $\P_r$ be the image by $\mathfrak{S}_2$ of the Schur process $\tilde{\P}_{e^{-r}}$ defined in section \ref{sec:intropp}. We start by stating the lemmas we use for the control of the covariances for $\P_r$, distinguishing the situations where the considered positions are far away from each other or not.
\begin{lem}\label{lemdecK} Let $\mathcal{K} \subset \R^2$ be a compact set, and let $m \subset \E$ be finite. Let $\overline{m}$ denote the supremum norm of $m$  \begin{align*}
\overline{m}=\max \{ |t|,|h|, \hspace{0.1cm} (t,h) \in m \}.
\end{align*}
Then for any $\delta\in (0,1)$, there exists $C$ which only depends on $\delta$, $\mathcal{K}$ and $m$, such that for all $r >0$ sufficiently small and any $(\tau_1,\chi_1),(\tau_2,\chi_2) \in A\cap K$ such that 
\begin{align} \label{conddist}
\max \{ |\tau_1-\tau_2|, |\chi_1-\chi_2| \} > \overline{m}r,
\end{align}
one has 
\begin{align*}
\left| \E_r \left[ c_{\frac{1}{r}(\tau_1,\chi_1)+m}c_{\frac{1}{r}(\tau_2,\chi_2)+m} \right]-\E_r \left[ c_{\frac{1}{r}(\tau_1,\chi_1)+m}\right]\E_r\left[c_{\frac{1}{r}(\tau_2,\chi_2)+m} \right] \right | \leq \frac{C\exp\left(-r^{-\delta}\right)}{|\tau_1-\tau_2|^2}
\end{align*}
when $\tau_1 \neq \tau_2$, and 
\begin{align*}
\left| \E_r \left[ c_{\frac{1}{r}(\tau,\chi_1)+m}c_{\frac{1}{r}(\tau,\chi_2)+m} \right]-\E_r \left[ c_{\frac{1}{r}(\tau,\chi_1)+m}\right]\E_r\left[c_{\frac{1}{r}(\tau,\chi_2)+m} \right] \right | \leq \frac{Cr}{|\chi_1-\chi_2|}
\end{align*}
when $\tau_1=\tau_2=\tau$.
\end{lem}
%We have the following key lemma for the decay of correlations for the limiting process :
%\begin{lem} \label{lemdecS} Let $(\tau,\chi) \in A$ and write $z=z(\tau,\chi)=|z|e^{i\phi}$. There exists $C>0$ such that for every $ \Delta t=t_1-t_2 \in \mathbb{Z}$ , $ \Delta h=h_1-h_2  \in \frac{1}{2}\mathbb{Z}$
%\begin{align*}
 %|\mathcal{S}_z(\Delta t , \Delta h )| \leq |1-z|^{\Delta t} |z|^{-\Delta h -\frac{\Delta t}{2} } \frac{C}{|\Delta t + \Delta h| +1}
%\end{align*}
%\end{lem}
%The last result we need is the following counting lemma :
%\begin{lem} \label{lemcount} For any $R>0$, one has :
%\begin{align*}
%\sum_{\substack{ t_1, t_2, h_1, h_2 \in \Z, \\ \max(|t_1|,|t_2|,|h_1|,|h_2|) \leq R}} \frac{1}{(|t_2-t_1|+|h_2-h_1|+1)^2 } \asymp R^3.
%\end{align*}
%\end{lem}
The next lemma we need is obtained as a simple corollary of Proposition \ref{propcont} below.
\begin{lem}\label{lemdiag}
For any compact $K \subset \R$ and any finite subsets $m,m' \subset E$, there exists $C$ such that for any $(\tau,\chi) \in A\cap K$, and any sufficiently small $r>0$ 
\begin{align*}
\left| \E_r \left[ c_{\frac{1}{r}(\tau,\chi)+m} c_{\frac{1}{r}(\tau,\chi)+m'}\right] - \E_r \left[ c_{\frac{1}{r}(\tau,\chi)+m}\right] \E_r \left[ c_{\frac{1}{r}(\tau,\chi)+m'} \right] \right| \leq C. 
\end{align*}
\end{lem}
\subsection{Proof of Theorem \ref{thm:ppp}} \label{sec:proofllnpp}
Let $\mathcal{K} \subset \R^2$ be a compact containing the support of $f$. We denote by $A_{r,\mathcal{K}}$ the set 
\begin{align*}
A_{r,\mathcal{K}}=r^{-1}(A \cap \mathcal{K} ) \cap E.
\end{align*}
By Proposition \ref{prop:cvpp}, there exists $C>0$ such that
\begin{align*}
 \left| \E_r \Sigma( f,m,r) - r^2 \sum_{(t,h) \in A_{r,\mathcal{K}}} f(rt,rh)  \E_{(rt,rh)} \left( c_{m} \right) \right| \lesssim C |A_{r,\mathcal{K}}| r^{3},
\end{align*}
where $|A_{r,\mathcal{K}}|$ is the cardinality of $A_{r,\mathcal{K}}$. We first remark that
\begin{align} \label{estimArK}
|A_{r,\mathcal{K}}| = O \left( r^{-2} \right),
\end{align}
which implies 
\begin{align*}
\E_{r} \Sigma (f,m,r)= r^2 \sum_{(t,h) \in A_{r,\mathcal{K}}} f(rt,rh) \E_{(rt,rh)} \left(c_{m}\right) + O(r).
\end{align*}
Observing then that 
\begin{align*}
r^2 \sum_{(t,h) \in A_{r,\mathcal{K}}} f(rt,rh) \E_{(rt,rh)} \left(c_{m}\right)
\end{align*}
is a Riemann sum for the integral $I(f,m)$, we obtain that 
\begin{align*}
\lim_{r \rightarrow 0} \E_{r} \Sigma(f,m,r) = I(f,m).
\end{align*}
By the Chebyshev inequality, it suffices now to prove that 
\begin{align} \label{var0}
\text{Var}_{r} \left( \Sigma(f,m,r) \right) \rightarrow 0 
\end{align}
as $r$ tends to $0$, where 
\begin{multline} \label{variance}
\text{Var}_{r} \left( \Sigma(f,m,r) \right) = \E_r \left[ \left(\Sigma(f,m,r) - \E_r \Sigma(f,m,r) \right)^2 \right] \\
=r^4 \sum_{(t_1h_1),(t_2,h_2) \in A_{r,\mathcal{K}}}  f(rt_1,rh_1)f(rt_2,rh_2) \\
\times \left(\E_r ( c_{(t_1,h_1)+m}c_{(t_2,h_2)+m} ) - \E_r ( c_{(t_1,h_1)+m} ) \E_r ( c_{(t_2,h_2)+m} ) \right).
\end{multline}
We set $ \overline{m}=\max\{ |t|,|h|, \hspace{0.1cm} (t,h) \in m \}$, and we partition $A_{r,\mathcal{K}}^2$ into three sets 
\begin{align*}
A_{r,\mathcal{K}}^2=A_{r,\mathcal{K}}^{>} \sqcup A_{r,\mathcal{K}}^{>=} \sqcup A_{r,\mathcal{K}}^{\leq},
\end{align*}
where :
\begin{align*}
A_{r,\mathcal{K}}^{>} &=\left\{ (t_1,h_1), (t_2,h_2) \in A_{r,\mathcal{K}}, \hspace{0.1cm} \max\{ |t_1-t_2|, |h_1-h_2|\} > \overline{m}, \hspace{0.1cm} t_1 \neq t_2 \right\}, \\
A_{r,\mathcal{K}}^{>=} &= \left\{ (t,h_1), (t,h_2) \in A_{r,\mathcal{K}}, \hspace{0.1cm} |h_1-h_2| > \overline{m} \right\},\\
A_{r,\mathcal{K}}^{\leq} &= A_{r,\mathcal{K}}^2 \setminus \left(A_{r,\mathcal{K}}^{>} \sqcup A_{r,\mathcal{K}}^{>=} \right) = \left\{ (t_1,h_1),(t_2,h_2) \in A_{r,\mathcal{K}}, \hspace{0.1cm} \max\{ |t_1-t_2|, |h_1-h_2|\} \leq \overline{m} \right\}.
\end{align*}
We first estimate the variance (\ref{variance}) by 
\begin{multline} \label{estimvar}
\text{Var}_{r} \left( \Sigma(f,m,r) \right) 
\leq C r^4 \left( \sum_{\left((t_1h_1),(t_2,h_2)\right) \in A_{r,\mathcal{K}}^>}  \left|\E_r ( c_{(t_1,h_1)+m}c_{(t_2,h_2)+m} ) - \E_r ( c_{(t_1,h_1)+m} ) \E_r ( c_{(t_2,h_2)+m} ) \right| \right. \\
+ \sum_{ \left((t,h_1),(t,h_2)\right) \in A_{r,\mathcal{K}}^{>=} } \left|\E_r ( c_{(t,h_1)+m}c_{(t,h_2)+m} ) - \E_r ( c_{(t,h_1)+m} ) \E_r ( c_{(t,h_2)+m} ) \right| \\
\left. + \sum_{\left((t_1h_1),(t_2,h_2)\right) \in A_{r,\mathcal{K}}^{\leq}}  \left|\E_r ( c_{(t_1,h_1)+m}c_{(t_2,h_2)+m} ) - \E_r ( c_{(t_1,h_1)+m} ) \E_r ( c_{(t_2,h_2)+m} ) \right|\right),
\end{multline}
where $C$ only depends on $f$. Let $(t_1,h_1),(t_2,h_2) \in A_{r,\mathcal{K}}$. By definition, there exists $(\tau_1,\chi_1),(\tau_2,\chi_2) \in A \cap K \cap rE$ such that :
\begin{align*}
(t_1,h_1)=\frac{1}{r}(\tau_1,\chi_1), \hspace{0.1cm} (t_2,h_2) =\frac{1}{r}(\tau_2,\chi_2).
\end{align*}
We first consider the case when $\left((t_1,h_,),(t_2,h_2)\right) \in A_{r,\mathcal{K}}^>$. The corresponding points $(\tau_1,\chi_1),(\tau_2,\chi_2)$ satisfy condition (\ref{conddist}), and thus by Lemma \ref{lemdecK}, we have in particular the estimate 
\begin{align*}
\left| \E_r \left[ c_{\frac{1}{r}(\tau_1,\chi_1)+m}c_{\frac{1}{r}(\tau_2,\chi_2)+m} \right]-\E_r \left[ c_{\frac{1}{r}(\tau_1,\chi_1)+m}\right]\E_r\left[c_{\frac{1}{r}(\tau_2,\chi_2)+m} \right] \right | \leq Cr,
\end{align*}
where $C$ is uniform. Since 
\begin{align*}
A_{r,\mathcal{K}}^> = O \left( r^{-4} \right),
\end{align*}
we obtain that 
\begin{align} \label{estimcov1}
\sum_{\left((t_1h_1),(t_2,h_2)\right) \in A_{r,\mathcal{K}}^>}  \left|\E_r ( c_{(t_1,h_1)+m}c_{(t_2,h_2)+m} ) - \E_r ( c_{(t_1,h_1)+m} ) \E_r ( c_{(t_2,h_2)+m} ) \right| \leq Cr^{-3}
\end{align}
where $C$ only depends on $\mathcal{K}$ and $m$. 
\par
In the case when $\left((t_1,h_,),(t_2,h_2)\right) \in A_{r,\mathcal{K}}^{>=}$, we have by Lemma \ref{lemdecK} that 
\begin{align*}
\left| \E_r \left[ c_{\frac{1}{r}(\tau_1,\chi_1)+m}c_{\frac{1}{r}(\tau_2,\chi_2)+m} \right]-\E_r \left[ c_{\frac{1}{r}(\tau_1,\chi_1)+m}\right]\E_r\left[c_{\frac{1}{r}(\tau_2,\chi_2)+m} \right] \right | \leq C
\end{align*}
where $C$ is uniform. Since 
\begin{align*}
|A_{r,\mathcal{K}}^{<=}| = O \left(r^{-3} \right), 
\end{align*}
we have 
\begin{align}\label{estimcov2}
\sum_{(t,h_1),(t,h_2) \in A_{r,\mathcal{K}}^{>=}}  \left|\E_r ( c_{(t,h_1)+m}c_{(t,h_2)+m} ) - \E_r ( c_{(t,h_1)+m} ) \E_r ( c_{(t,h_2)+m} ) \right| \leq Cr^{-3},
\end{align}
where $C$ only depends on $\mathcal{K}$ and $m$.\\

When $\left((t_1,h_1),(t_2;h_2)\right) \in A_{r,\mathcal{K}}^{\leq}$, there exists finite subsets $m',m'' \subset E$ and $(\tau,\chi) \in A\cap K$ such that 
\begin{align*}
(t_1,h_1)+m= \frac{1}{r}(\tau,\chi) +m', \hspace{0.1cm} (t_2,h_2)+m = \frac{1}{r}(\tau,\chi) +m''.
\end{align*}
Observe that there is only a finite number of possible sets $m'$ and $m''$.Thus, by Lemma \ref{lemdiag}, we have 
\begin{align*}
\left| \E_r \left[ c_{\frac{1}{r}(\tau,\chi)+m'} c_{\frac{1}{r}(\tau,\chi)+m''}\right] - \E_r \left[ c_{\frac{1}{r}(\tau,\chi)+m'}\right] \E_r \left[ c_{\frac{1}{r}(\tau,\chi)+m''} \right] \right| \leq C
\end{align*}
where $C$ is uniform. Since 
\begin{align*}
|A_{r,\mathcal{K}}^{\leq} |= O \left( r^{-2} \right),
\end{align*}
we have 
\begin{align} \label{estimcov3}
\sum_{(t,h_1),(t,h_2) \in A_{r,\mathcal{K}}^{\leq} }\left|   \E_r ( c_{(t,h_1)+m}c_{(t,h_2)+m} ) - \E_r ( c_{(t,h_1)+m} ) \E_r ( c_{(t,h_2)+m} ) \right|  \leq Cr^{-2}
\end{align}
Thus, recalling the estimation (\ref{estimvar}), the inequalities (\ref{estimcov1}), (\ref{estimcov2}) and (\ref{estimcov3}) establish (\ref{var0}). Theorem \ref{thm:ppp} is proved, assuming Proposition \ref{prop:cvpp} and Lemmas \ref{lemdecK} and \ref{lemdiag}.
\subsection{Proof of Proposition \ref{prop:cvpp}} \label{sec:proofproplimpp}We here follow the proof of \cite{okounkovreshetikhin}, giving the error terms in the asymptotics we use. We define the dilogarithm function as being the analytic continuation of the series 
\begin{align*}
\text{dilog}(1-z)=\sum_{n \geq 1} \frac{z^n}{n^2}, \quad |z|<1,
\end{align*}
with a cut along the half-line $(1,+\infty)$. We first prove that 
\begin{align} \label{estimdilog}
-\log(z,q)_\infty = r^{-1}\text{dilog}(1-z) + O(1).
\end{align}
as $q=e^{-r}$ tends to $1^-$. Indeed, we have 
\begin{align*}
\log(z,q)_{\infty}&=\sum_{k \geq 0} \log(1-zq^k) = -\sum_{k \geq 0} \sum_{n \geq 1} \frac{z^nq^{nk}}{n} \\
&=-\sum_{n \geq 1} \frac{z^n}{n} \sum_{k \geq 0} q^{nk} =-\sum_{n \geq 0} \frac{z^n}{n} \frac{1}{1-q^n}.
\end{align*}
With $q=e^{-r}$, we have 
\begin{align*}
\frac{r}{1-e^{-nr}}=\frac{1}{n -n^2r + ...}=\frac{1}{n}(1+nr+...),
\end{align*}
and thus 
\begin{align*}
\left| \frac{z^n}{n}\left(\frac{1}{1-e^{-nr}}- \frac{1}{n}\right)\right| \leq rn|z|^n,
\end{align*}
which establishes (\ref{estimdilog}). 

Let $\mathcal{K} \subset \R^2$ be compact and let $(\tau,\chi) \in A\cap \mathcal{K}$. We assume that $\tau \geq 0$, see \ref{secrk} below for the case $\tau \leq 0$.  We introduce the function 
\begin{align*}
S(z;\tau,\chi)=-(\tau/2+ \chi)\log(z)-\dilog(1-1/z)+ \dilog(1-e^{-\tau}z),
\end{align*}
and denote by  $\gamma_\tau$ the circle 
\begin{align*}
\gamma_\tau = \{ z \in \C, \hspace{0.1cm} |z|=e^{\tau/2} \}.
\end{align*}
By the estimate (\ref{estimdilog}) and formula (\ref{defKq}), we have that, for all $z$ and $w$ sufficiently closed to $\gamma_\tau$ 
\begin{multline} \displaystyle
\left| \frac{\Phi(\tau/r+ t_1,z)}{\Phi(\tau/r+t_2,w)}\frac{1}{z^{\chi/r+h_1+\tau/2r+(t_1+1)/2}w^{-\chi/r-h_2-\tau/2r-(t_2+1)/2}} \right| \\
=  \exp \left( \frac{1}{r} \left( \mathfrak{R} S(z;\tau,\chi)-\mathfrak{R}S(w;\tau,\chi) \right) + O(1)  \right) ,
\end{multline}
where the $O(1)$ term only depens on $\mathcal{K}$, $(t_1,h_1)$ and $(t_2,h_2)$. An observation made in \cite{okounkovreshetikhin} states that the real part of $S$ on the circle $\gamma_\tau$ is constant, namely 
\begin{align} \label{reS}
\mathfrak{R}S(z;\tau,\chi)=-\frac{\tau}{2}(\tau/2+\chi),\quad z \in \gamma_\tau.
\end{align}
It is also shown in \cite{okounkovreshetikhin} that, since $(\tau,\chi) \in A$, the function $S$ has two distinct critical points on $\gamma_\tau$ : $e^{\tau}z(\tau,\chi)$ and its complex conjugate. The computation of the gradient of the real part of $S$ on $\gamma_\tau$ lead then the authors of \cite{okounkovreshetikhin} to deform the circle $\gamma_\tau$ into simple contours $\gamma_\tau^>$ and $\gamma_\tau^<$, both crossing the two critical points and verifying 
\begin{equation}
\begin{split}
z \in \gamma_\tau^> \Rightarrow \mathfrak{R}S(z;\tau,\chi) \geq -\frac{\tau}{2}(\tau/2+\chi), \\
z \in \gamma_\tau^< \Rightarrow \mathfrak{R}S(z;\tau,\chi) \leq -\frac{\tau}{2}(\tau/2+\chi),
\end{split}
\end{equation}
with equality only for $z\in \left\{e^\tau z(\tau,\chi), e^\tau \overline{z(\tau,\chi)} \right\}$, see Figure 6.
\begin{figure}[!h] \label{fig1}
\centering
\includegraphics[scale=0.5,clip=true,trim=0cm 0cm 0cm 0cm]{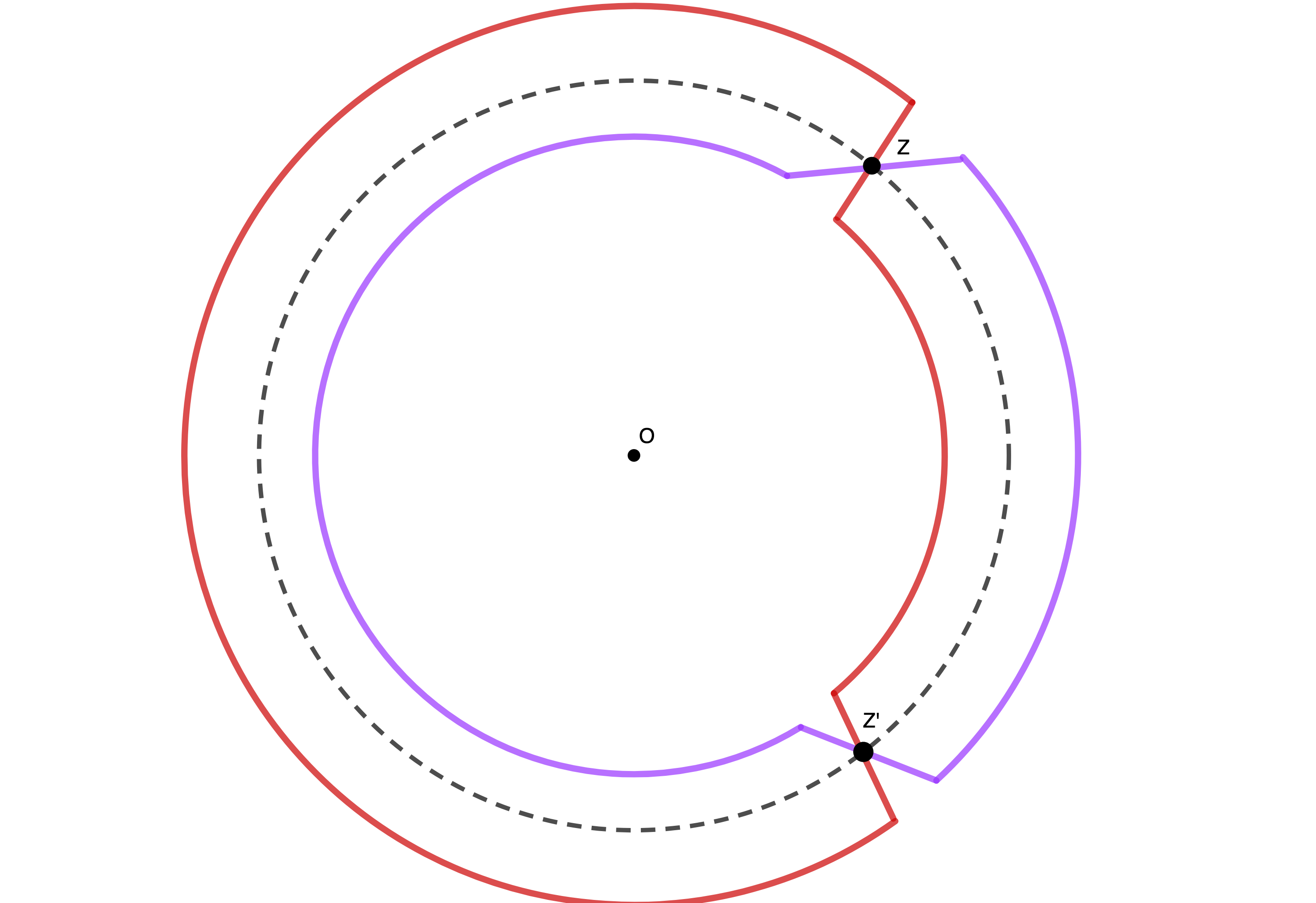}
\caption{The contours $\gamma_\tau^>$ and $\gamma_\tau^<$ are the thick contours and the circle $\gamma_\tau$ is the dotted circle.}
\end{figure}
These simple facts imply that the integral 
\begin{align} \label{intdef}
 \int_{z \in \gamma_\tau^<} \int_{w \in \gamma_\tau^>} \exp \left( \frac{1}{r} \left( S(z;\tau,\chi)-S(w;\tau,\chi) \right)\right) \frac{dzdw}{z-w}
\end{align}
goes to zero as $r$ tends to zero. Actually, the dominated convergence theorem implies that the integral (\ref{intdef}) is $O\left(\exp\left(-r^{-\delta}\right)\right)$ for any $\delta \in (0,1)$.

Picking the residue at $z=w$, we arrive at 
\begin{multline} \label{sumKq}
\mathcal{K}_{e^{-r}} \left( \frac{\tau}{r} + t_1, \frac{\chi}{r} +h_1 ; \frac{\tau}{r} + t_2, \frac{\chi}{r} +h_2\right)\\
 = \frac{1}{(2i\pi)^2} \int_{z \in \gamma_\tau^<} \int_{w \in \gamma_\tau^>} \exp \left( \frac{1}{r} \left( S(z;\tau,\chi)-S(w;\tau,\chi) \right)+ O(1) \right) \frac{dzdw}{z-w} \\
+ \frac{1}{2i\pi} \int_{e^\tau \overline{z(\tau,\chi)}}^{e^\tau z(\tau,\chi)} \frac{(q^{1/2+\tau/r+t_2}w;q)_\infty}{(q^{1/2+ \tau/r +t_1};q)_\infty} \frac{dw}{w^{h_1-h_2+(t_1-t_2)/2}},
\end{multline}
where the path of integration for the second integral crosses the interval $(0,e^\tau)$ for $t_1\geq t_2$ and the half-line $(-\infty,0)$ otherwise. By the preceding discussion, the first integral rapidly tends to zero. Observe now that 
\begin{align*}
\frac{(q^{1/2+\tau/r+t_2}w;q)_\infty}{(q^{1/2+ \tau/r +t_1};q)_\infty} = (1+O(r))\left(1-e^{-\tau}w\right)^{t_1-t_2},
\end{align*}
where the $O(r)$ term only depends on $\mathcal{K}$, $t_1$ and $t_2$. Performing the change of variable $w \mapsto e^{-\tau}w$ in the second integral of (\ref{sumKq}), we arrive at 
\begin{multline*} 
\frac{1}{2i\pi} \int_{e^\tau \overline{z(\tau,\chi)}}^{e^\tau z(\tau,\chi)} \frac{(q^{1/2+\tau/r+t_2}w;q)_\infty}{(q^{1/2+ \tau/r +t_1};q)_\infty} \frac{dw}{w^{h_1-h_2+(t_1-t_2)/2}} \\
= \left( 1 +O(r) \right) e^{-\tau(h_1-h_2-(t_1-t_2)/2)} \mathcal{S}_{\tau,\chi }( t_1-t_2,h_1-h_2).
\end{multline*}
The factor $e^{-\tau(h_1-h_2-(t_1-t_2)/2)}$ can be ignored, see Remark \ref{rem:gaugetransform}. Proposition \ref{prop:cvpp} is proved.
\subsection{Proof of Lemmas \ref{lemdecK} and  \ref{lemdiag}} \label{sec:proofcovpp}
\subsubsection{A remark and a proposition} \label{secrk}
For $\tau <0$, one has to replace the function $S$ by 
\begin{align*}
\tilde{S}(z,\tau,\chi)=-(|\tau|/2+\chi)\log(z)-\dilog(1-z)+\dilog(1-e^{-|\tau|}/z).
\end{align*}
The function $\tilde{S}$ innerhits the same properties than the function $S$: it is constant on the circle $\gamma_{|\tau|}$ and has two complex conjugated critical points on this circle provided $(\tau,\chi) \in A$. This is why we will only consider positve values of $\tau$ in the sequel.\\

The critical points of $S$ are the roots of the quadratic polynomial 
\begin{align*}
(1-1/z)(1-e^{-\tau}z)=e^{-\tau/2-\chi}.
\end{align*}
For this reason, we have the following proposition :
\begin{prop} \label{propcont}
For any fixed $(\Delta t, \Delta h) \in E$, the function :
\begin{align*}
(\tau,\chi) \mapsto \mathcal{S}_{\tau,\chi}(\Delta t, \Delta h)
\end{align*}
is continuous.
\end{prop}
\subsubsection{Proof of lemma \ref{lemdecK}}
Let $m \subset E$ be a pattern, of cardinality $l$, let $\mathcal{K} \subset \R^2$ be a compact set and let $(\tau_1,\chi_1), (\tau_2,\chi_2) \in A\cap \mathcal{K}$ be as in the statement of the lemma. The condition (\ref{conddist}) implies that the sets $\frac{1}{r}(\tau_1,\chi_1)+m$ and $\frac{1}{r}(\tau_2,\chi_2)+ m$ are disjoints. By Proposition \ref{prop:covdet}, the covariance
\begin{align*}
\mathrm{cov}_r(c_{ \frac{1}{r}(\tau_1,\chi_1) + m}, c_{m \frac{1}{r} (\tau_2,\chi_2) + m}) = \E_r [c_{ \frac{1}{r}(\tau_1,\chi_1) + m} c_{m \frac{1}{r} (\tau_2,\chi_2) + m}] - \E_r [c_{ \frac{1}{r}(\tau_1,\chi_1) + m} ] \E_r [c_{m \frac{1}{r} (\tau_2,\chi_2) + m}] 
\end{align*}
is a sum of $(2l)! - (l!)^2$ terms, each of them containing a factor of the form
\begin{align} \label{prodK}
\K_{e^{-r}}\left(\frac{1}{r}(\tau_1,\chi_1)+m_{i_1}; \frac{1}{r}(\tau_2,\chi_2)+m_{j_1}\right)\K_{e^{-r}}\left(\frac{1}{r}(\tau_2,\chi_2)+m_{j_2}; \frac{1}{r}(\tau_1,\chi_1)+m_{i_2}\right).
\end{align}
By Propositions \ref{prop:cvpp} and \ref{propcont}, the other factors are bounded by a bound only depending on $\mathcal{K}$ and $m$. By similar methods as in the proof of Proposition \ref{prop:cvpp}, we will show that the product (\ref{prodK}) is small. 

The product (\ref{prodK}) can be written as a quadruple integral 
\begin{multline*} \displaystyle
\K_{e^{-r}}\left(\frac{1}{r}(\tau_1,\chi_1)+m_{i_1}; \frac{1}{r}(\tau_2,\chi_2)+m_{j_1}\right)\K_{e^{-r}}\left(\frac{1}{r}(\tau_2,\chi_2)+m_{j_2}; \frac{1}{r}(\tau_1,\chi_1)+m_{i_2}\right) \\
= \frac{1}{(2i\pi)^4} \int_{z \in (1+\varepsilon)\gamma_{\tau_1}} dz\int_{w \in (1-\varepsilon)\gamma_{\tau_2}} dw\int_{z' \in (1-\varepsilon)\gamma_{\tau_2}}dz' \int_{w' \in (1+\varepsilon)\gamma_{\tau_1}} dw' \\
\frac{\exp\left(\frac{1}{r}\left(S(z;\tau_1,\chi_1)-S(w;\tau_2,\chi_2)+S(z';\tau_2,\chi_2)-S(w';\tau_1,\chi_1)\right)+O(1)\right)}{(z-w)(z'-w')}
\end{multline*}
We first consider the case when $\tau_1 \neq \tau_2$, and by symmetry, we assume that $\tau_1 > \tau_2$. One can then deform the contours as previously. Precisely, we now integrate over 
\begin{align*}
z \in \gamma_{\tau_1}^<, \hspace{0.1cm} w \in \gamma_{\tau_2}^>, \hspace{0.1cm} z' \in \gamma_{\tau_2}^<, \hspace{0.1cm} w' \in \gamma_{\tau_1}^>
\end{align*}
in order to have
\begin{align*}
\mathfrak{R}\left(S(z;\tau_1,\chi_1)-S(w';\tau_1,\chi_1) \right) <0, \quad \text{and} \quad \mathfrak{R}\left(S(z';\tau_2,\chi_2)-S(w;\tau_2,\chi_2) \right) <0,
\end{align*}
see Figure 7.
\begin{figure}[!h] \label{fig2}
\centering
\includegraphics[scale=0.3,clip=true,trim=0cm 0cm 0cm 0cm]{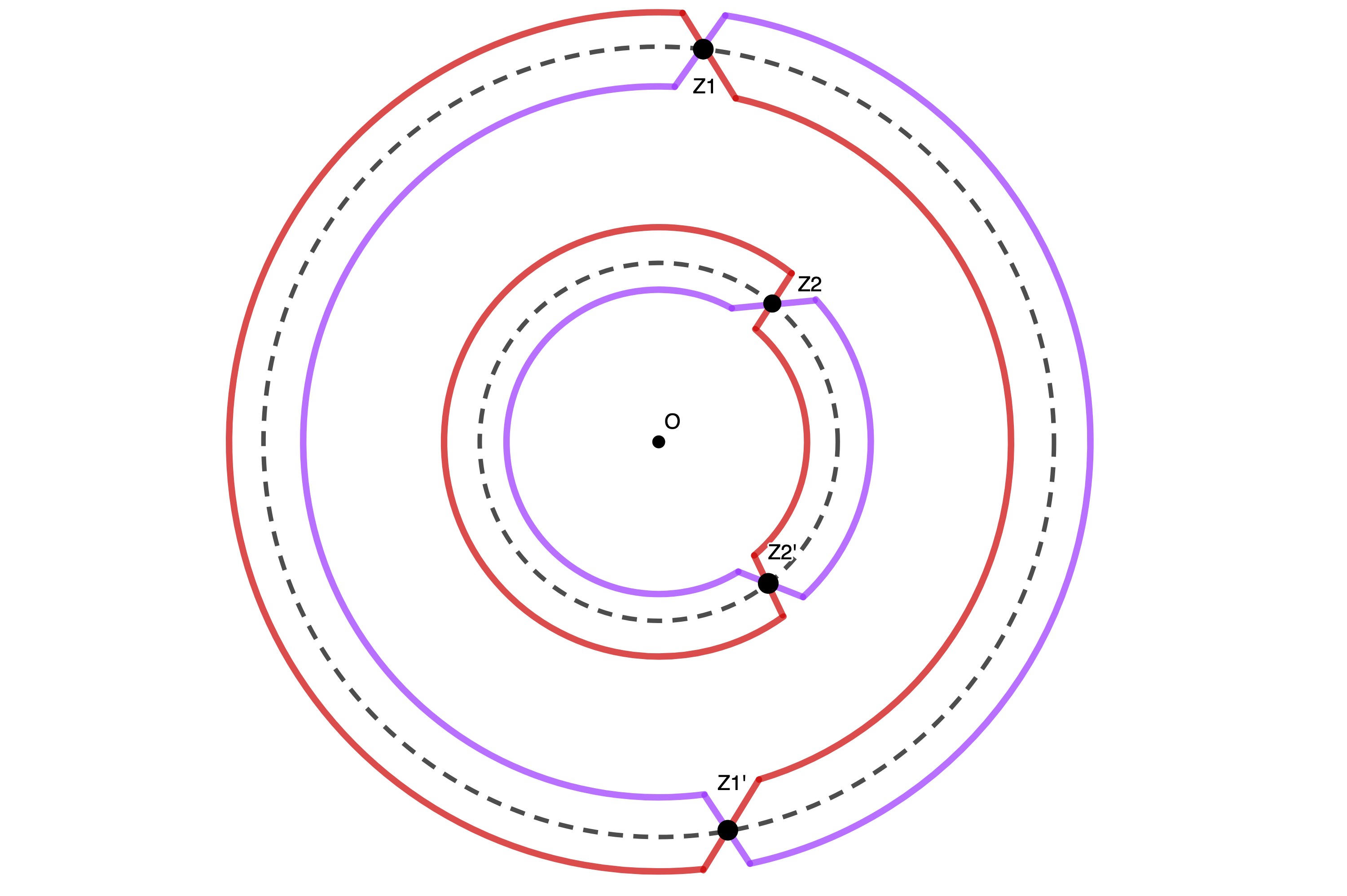}
\caption{The contours $\gamma_{\tau_1}^>$, $\gamma_{\tau_1}^<$ are the thick contours near the dotted circle with the largest radius, the circle $\gamma_{\tau_1}$; the contours $\gamma_{\tau_2}^>$ and $\gamma_{\tau_2}^<$ are the thick contours near the dotted circle with the smallest radius $\gamma_{\tau_2}$.}
\end{figure}
These deformations do not affect the value of the integrals, because they involve separate variables. Since, for all $\delta \in (0,1)$, we have 
\begin{align*}
\frac{\exp\left(\frac{1}{r}\left(S(z;\tau_1,\chi_1)-S(w;\tau_2,\chi_2)+S(z';\tau_2,\chi_2)-S(w';\tau_1,\chi_1)\right)\right)}{\exp(r^{-\delta})} \rightarrow 0
\end{align*}
as $r \rightarrow 0$, for all $z,z',w,w'$ in the new contours except at a finite number of points, and since 
\begin{align*}
\left| \frac{\exp\left(\frac{1}{r}\left(S(z;\tau_1,\chi_1)-S(w;\tau_2,\chi_2)+S(z';\tau_2,\chi_2)-S(w';\tau_1,\chi_1)\right)\right)}{(z-w)(z'-w')} \right| \leq \frac{1}{|e^{\tau_1}-e^{\tau_2}|^2}
\end{align*}
for all $z,z',w,w'$ in the new contours except at a finite number of points, the for the case $\tau_1 >\tau_2$ is complete.\\

For $\tau_1=\tau_2=\tau$, the proof is as follows. One deform the contours as for the preceding case, but now, the deformations affect the value of the kernel since we can avoid the residues at $z=w$ and $z'=w'$. We have for example the following case 
\begin{equation} \label{prodKq2}
\begin{split}
\K&_{e^{-r}}\left(\frac{1}{r}(\tau,\chi_1)+(t_1^1,h_1^1); \frac{1}{r}(\tau,\chi_2)+(t_2^1,h_2^1)_1\right)\K_{e^{-r}}\left(\frac{1}{r}(\tau,\chi_2)+(t_1^2,h_1^2)_2; \frac{1}{r}(\tau,\chi_1)+(t_2^2,h_2^2)_2\right) \\
&=(1+O(1))\left( \frac{1}{(2i\pi)^2}\int_{z \in \gamma_\tau^{<,1}}\int_{w \in \gamma_\tau^{>,2}}... +\frac{1}{2i\pi}\int_w Res_{z=w} f(z,w;\tau,\chi_1,\chi_2)dw \right) \\
&\times
\left(\frac{1}{(2i\pi)^2} \int_{z' \in \gamma_\tau^{<,2}}\int_{w' \in \gamma_\tau^{>,1}}... + \frac{1}{2i\pi}\int_{w'} Res_{z'=w'}f(z',w';\tau,\chi_2,\chi_1)dw' \right),
\end{split}
\end{equation} 
where 
\begin{multline} \label{res}
\int_w Res_{z=w}f(z,w;\tau,\chi_1,\chi_2)dw\\
=\int_{|w|=e^{\tau/2}, \hspace{0.1cm} |\arg(w)|<\phi_{\tau,\chi_2}} \frac{(q^{1/2+\tau/r+t_2}w;q)_\infty}{(q^{1/2+\tau/r+t_1}w;q)_\infty}\frac{dw}{w^{1/r(\chi_1-\chi_2)+h_1^1-h_2^1 + t_1^1-t_2^1}}, \\
\int_{w'} Res_{z'=w'}f(z',w';\tau,\chi_2,\chi_1)dw\\
=\int_{|w'|=e^{\tau/2}, \hspace{0.1cm} |\arg(w')|<\phi_{\tau,\chi_2}} \frac{(q^{1/2+\tau/r+t_2}w';q)_\infty}{(q^{1/2+\tau/r+t_1}w';q)_\infty}\frac{dw'}{w'^{1/r(\chi_2-\chi_1)+ h_1^2-h_2^2 +  t_1^2-t_2^2}},
\end{multline}
the argument $\phi_{\tau,\chi_2}$ being an argument of $z(\tau,\chi_2)$, see Figure 8. Equality (\ref{prodKq2}) is valid when 
\begin{align*}
t_1^1 \geq t_2^1, \hspace{0.1cm} t_1^2 \geq t_2^2, \hspace{0.1cm} \text{and } \arg \left( z(\tau,\chi_1) \right) <\arg \left( z(\tau,\chi_2) \right),
\end{align*}
and the other cases can be treated in a similar way.
\begin{figure}[!h]
\centering
\includegraphics[scale=0.5,clip=true,trim=0cm 0cm 0cm 0cm]{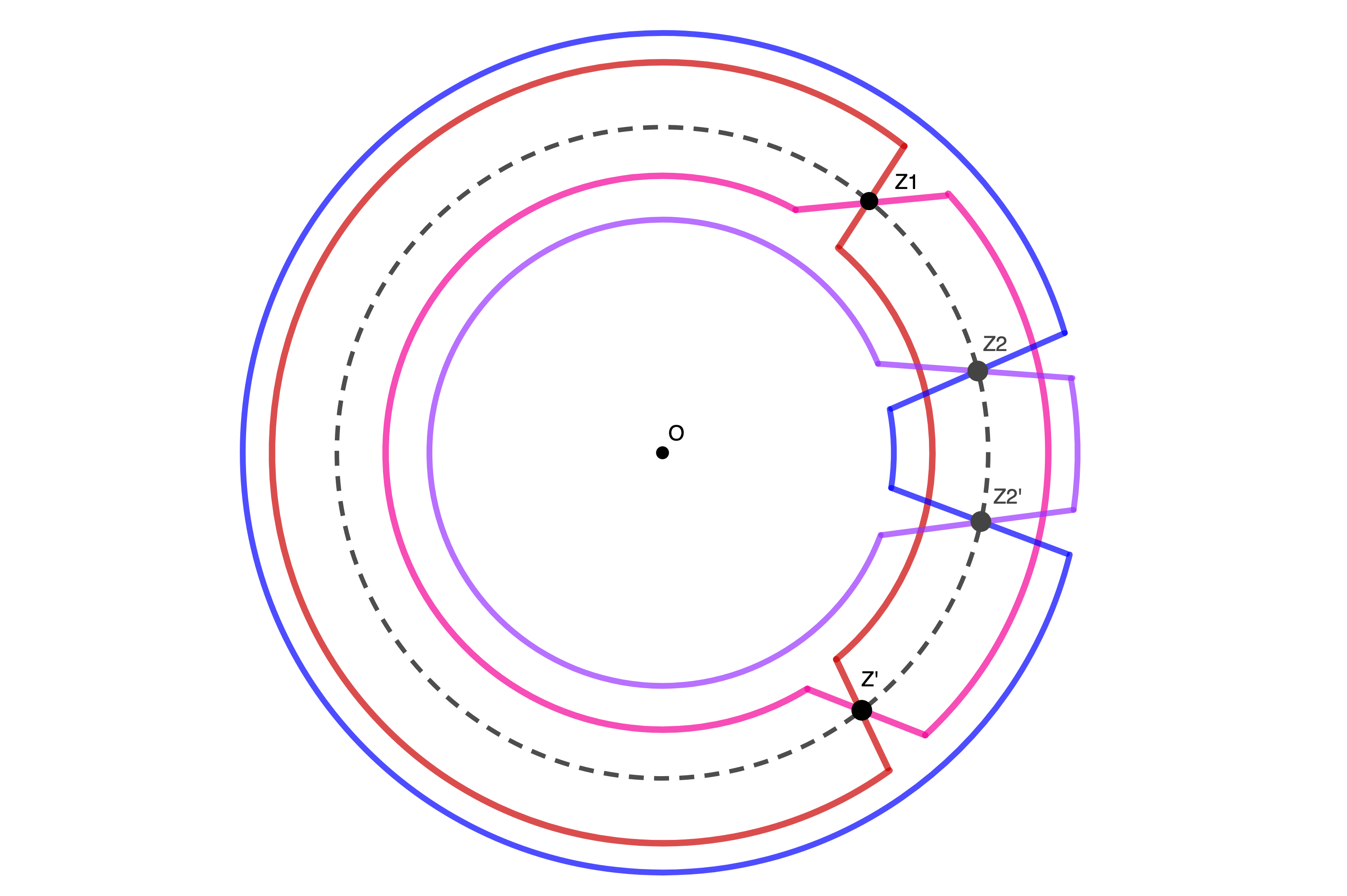}
\caption{The thick contours $\gamma_{\tau}^{>,1}$ and $\gamma_\tau^{<,1}$ cross the dotted circle $\gamma_\tau$ at points $z1=e^\tau z(\tau,\chi_1)$ and $z1'=e^\tau\overline{z(\tau,\chi_1)}$, while the thick contours $\gamma_{\tau}^{>,2}$ and $\gamma_{\tau}^{<,2}$ cross the circle $\gamma_\tau$ at points $z2=e^\tau z(\tau,\chi_2)$ and $z2'=e^\tau \overline{z(\tau,\chi_2)}$.}
\end{figure}
Note that the factor 
\begin{align*}
f_r(w):=\frac{(q^{1/2+\tau/r+t_2}w;q)_\infty}{(q^{1/2+\tau/r+t_1}w;q)_\infty}
\end{align*}
is bounded, since it tends to 
\begin{align*}
(1-e^{-\tau}w)^{t_1-t_2}
\end{align*}
as $r$ tends to $0$. Integrating (\ref{res}) by parts leads to 
\begin{multline} \label{ineqres}
\left| \int_w Res_{z=w}f(z,w;\tau,\chi_1,\chi_2)dw \right| \\
\leq  C \frac{\exp\left(\frac{\tau}{2r}(\chi_2-\chi_1)\right)}{1/r|\chi_1-\chi_2|}
+\frac{\exp\left(\frac{\tau}{2r}(\chi_2-\chi_1)\right)}{1/r|\chi_1-\chi_2|}\left| \int_{|w|=1, \hspace{0.1cm} \arg(w)<\phi_{\tau,\chi_2}}f_r'(e^{-\tau/2}w)\frac{dw}{w^{1/r(\chi_1-\chi_2)+\Delta h + \Delta t-1}}\right| \\
 \leq C\frac{r}{|\chi_1-\chi_2|}\exp \left( \frac{\tau}{2r}(\chi_2-\chi_1)\right),
\end{multline}
and 
\begin{equation} \label{ineq2}
\left| \int_{w'} Res_{z'=w'}f(z',w';\tau,\chi_2,\chi_1)dw \right|  \leq C\frac{r}{|\chi_1-\chi_2|}\exp \left( \frac{\tau}{2r}(\chi_1-\chi_2)\right).
\end{equation}
It is clear that, by construction, we have 
\begin{equation}\label{ineqint}
\left|\int_{z \in \gamma_\tau^{<,1}}\int_{w \in \gamma_\tau^{>,2}}... \right| \leq C \exp \left( \frac{\tau}{2r}(\chi_2-\chi_1)\right),
\end{equation}
and 
\begin{equation}\label{ineqint2}
\left|\int_{z' \in \gamma_\tau^{<,2}}\int_{w' \in \gamma_\tau^{>,1}}... \right| \leq C \exp \left( \frac{\tau}{2r}(\chi_1-\chi_2)\right).
\end{equation}
We now expand the product in (\ref{prodKq2}). The term 
\begin{align*}
\int_{z \in \gamma_\tau^{<,1}}\int_{w \in \gamma_\tau^{>,2}}... \times \int_{z' \in \gamma_\tau^{<,2}}\int_{w' \in \gamma_\tau^{>,1}}...
\end{align*}
is by construction dominated by any polynomial in $r$. The estimates (\ref{ineqres}) and (\ref{ineq2}) imply that the product of the integrals of the residues is smaller than 
\begin{align*}
\frac{Cr^2}{|\chi_1-\chi_2|^2},
\end{align*}
while the combinations of (\ref{ineqres}) and (\ref{ineqint2}), and (\ref{ineq2}) and (\ref{ineqint}) entail that the remaining terms are smaller than 
\begin{align*}
\frac{Cr}{|\chi_1-\chi_2|}.
\end{align*}
Lemma \ref{lemdecK} is proved.
\subsubsection{Proof of Lemma \ref{lemdiag}}
Lemma \ref{lemdiag} is proved using Propositions \ref{prop:cvpp} and \ref{propcont}. Let $\mathcal{K} \subset \R^2$ be a compact set, and let $m,m' \subset E$ be finite. Let $(\tau,\chi) \in A\cap \mathcal{K}$. By Proposition \ref{prop:cvpp}, we have
\begin{multline*}
 \E_r \left[ c_{\frac{1}{r}(\tau,\chi)+m} c_{\frac{1}{r}(\tau,\chi)+m'}\right] - \E_r \left[ c_{\frac{1}{r}(\tau,\chi)+m}\right] \E_r \left[ c_{\frac{1}{r}(\tau,\chi)+m'} \right] \\
 =  \E_{(\tau,\chi)} \left[ c_{m} c_{m'}\right] - \E_{(\tau,\chi)} \left[ c_m\right] \E_{(\tau,\chi)} \left[ c_{m'} \right]+O(r).
\end{multline*} 
Now, by Proposition \ref{propcont}, the function 
\begin{align*}
(\tau,\chi) \mapsto  \E_{(\tau,\chi)} \left[ c_{m} c_{m'}\right] - \E_{(\tau,\chi)} \left[ c_m\right] \E_{(\tau,\chi)} \left[ c_{m'} \right]
\end{align*}
is bounded, as long as $(\tau, \chi)$ belong to a compact set. Lemma \ref{lemdiag} is proved, and Theorem \ref{thm:ppp} is completely proved.

\end{document}